\documentclass[12pt]{article}
\usepackage{amsthm}
\usepackage{amsmath}
\usepackage{amssymb}
\usepackage{graphicx}
\usepackage{enumerate}
\usepackage{subfig}

\newtheorem{thm}{Theorem}[section]
\newtheorem{rem}{Remark}[section]
\newtheorem{prop}{Proposition}[section]

\begin{document}

\begin{center}
\bf Chaos in Dynamics of a Family of Transcendental Meromorphic Functions
\end{center}

\begin{center}
M. Sajid* and G. P. Kapoor**
\end{center}
\begin{center}
*College of Engineering, Qassim University, Buraidah, Saudi Arabia\\
Email: msajid@qec.edu.sa \\
**Department of Mathematics, Indian Institute of Technology Kanpur, India\\
Email: gp@iitk.ac.in
\end{center}

\begin{abstract}
The characterization and properties of Julia sets of one parameter family of transcendental
meromorphic functions $\zeta_\lambda(z)=\lambda
\frac{z}{z+1} e^{-z}$, $\lambda >0$, $z\in \mathbb{C}$ is investigated in the present paper.
It is found that bifurcations in the dynamics of $\zeta_\lambda(x)$,
$x\in {\mathbb{R}}\setminus \{-1\}$, occur at several parameter
values and the dynamics of the family becomes chaotic when the parameter $\lambda$ crosses certain values. The Lyapunov exponent of $\zeta_\lambda(x)$ for certain values of the parameter $\lambda$ is computed for quantifying the chaos in its dynamics.
 The characterization of the Julia set of the function $\zeta_\lambda(z)$
as complement of the basin of attraction of an attracting real fixed
point of $\zeta_\lambda(z)$ is found here and is applied to computationally simulate the images of the Julia sets of
$\zeta_\lambda(z)$. Further, it is established that the Julia set of
$\zeta_\lambda(z)$ for $\lambda>(\sqrt{2}+1) e^{\sqrt{2}}$ contains
the complement of attracting
periodic orbits of $\zeta_\lambda(x)$. Finally, the results on the dynamics of
functions $\lambda \tan z$, $\lambda \in
{\mathbb{\hat{C}}}\setminus\{0\}$, $E_{\lambda}(z) = \lambda \frac{e^{z} -1}{z}$, $\lambda > 0$
and $f_{\lambda}=\lambda f(z)$, $\lambda>0$, where $f(z)$ has certain properties, are compared with the results found in the present paper.
\end{abstract}

\section{Introduction}
The dynamical system $(X, f_{\lambda})$, $\lambda \in {\mathbb{R}}$, when the function $f_{\lambda}$ is in a certain class of polynomials, is investigated in~\cite{brannerhub92,cheng89,douadyhub82} to explore its global dynamical behaviour. This study is extended to the family of entire functions in~\cite{dev2001,dong92,gpkmgpp1,kurodajang97,yanagotoh98} that provides powerful mathematical techniques and beautiful computer graphics. Since the Julia sets occur as one of the crucial component in these investigations, their various characterizations and identification of their intrinsic properties are primarily developed in these studies. The work in this direction has also found applications in a number of diverse science and engineering areas wherein simulations of objects of fractal nature are needed~\cite{gontar,govin,nakagawa,psj}. 
However, similar investigations concerning the dynamics of
transcendental meromorphic functions have been scarcely made. Since, for
instance, the iterative processes associated with Newton's method
applied to an entire function often yields a meromorphic function as
the root finder. The initial study of iterations of transcendental meromorphic
functions is mainly found in~\cite{baker1,baker92,berg93,devkeen1}. Further work in this direction has been pursued in~\cite{doming,keenkotus97,zheng2kpre}.

\par
Let $S$ be the class of critically finite (having only finitely many singular values) transcendental meromorphic
function $f(z)$. Baker~\cite{baker92} proved that a function $f(z)$
in the class $S$ has no wandering domains:
\begin{thm}
Let $f \in S$.  Then, $f(z)$ has no  wandering domains.
\label{thm1114}
\end{thm}
Bergweiler~\cite{berg93} proved that a function $f \in S$ has no Baker domains. Thus,
\begin{thm}
Let $f \in S$. Then, $f(z)$ has no  Baker domains.
\label{thm1115}
\end{thm}
\par
The purpose of the present paper
is to investigate the dynamics of a one parameter family
of transcendental meromorphic functions $f_{\lambda}(z)$ that have rational
Schwarzian derivative $SD(f_{\lambda})=\biggl(\frac{f_{\lambda}''(z)}{f_{\lambda}'(z)}\biggr)'-\frac{1}{2}\biggl(\frac{f_{\lambda}''(z)}{f_{\lambda}'(z)}
\biggr)^{2}$, are critically finite and possess both
critical as well as asymptotic values. Despite the Schwarzian derivative in our
family being rational, these functions are shown to have
their dynamical behaviour somewhat similar to that of Devaney and Keen's function~\cite{devkeen1}, wherein the Schwarzian derivatives of functions in the family are just polynomial. However, certain basic differences emerge in the dynamics of our family of functions and the dynamics of family of functions considered by Devaney and Keen.
\par
The structure of the present paper is as follows: In
Section~\ref{fpnature}, the existence and nature of fixed points as
well as periodic points of $\zeta_{\lambda}(x)$ is described. The
dynamics of $\zeta_{\lambda}(x)$, $x \in {\mathbb{R}} \backslash
\{-1\}$, is investigated in Section~\ref{bifur3}. In this section,
it is found that $\lambda = 1$ and $\lambda=(\sqrt{2}+1) e^{\sqrt{2}}$ are the
bifurcation points of $\zeta_{\lambda}(x)$.  The Lyapunov exponent of $\zeta_{\lambda}(x)$ is computed and the values of the parameter $\lambda$ are determined for which its Lyapunov exponent is positive, thereby ascertaining chaos in the dynamics of the function for these values of the parameter. The chaotic behaviour
of complex functions in our family is described by Julia sets of
its functions. In Section~\ref{sec:a11}, the characterization of
Julia set $J({\zeta_{\lambda}})$, $0 <\lambda < 1$, as the
complement of the basin of attraction of an attracting real fixed
point of $\zeta_{\lambda}(z)$ is established. Further, in this
section, it is proved that Fatou set of the function
$\zeta_{\lambda}(z)$ for $\lambda = 1$ contains certain intervals of
real line. For $1<\lambda \leq \lambda^{*}$, the characterization of
Julia set $J(\zeta_{\lambda})$ and the property of Fatou set
of $\zeta_{\lambda}(z)$ is found in Section~\ref{sec:c11} and is
seen to be similar to that in the case $0<\lambda< 1$. Further, in
this section, it is shown that Julia set of $\zeta_{\lambda}(z)$
for $\lambda > \lambda^{*}$ contains the subset of real line
consisting of the complement of points attracted to the attracting periodic orbits. In
Section~\ref{algo44}, the characterization of Julia set of
$\zeta_{\lambda}(z)$ for different ranges of parameter value
$\lambda$, obtained in Sections~\ref{sec:a11} and~\ref{sec:c11}, are
used to computationally generate the images of its Julia sets. 
The results obtained in the present paper are finally
compared with some of the results found
in~\cite{devkeen1,devtanger,gpkmgpp1,msgpk,stallard94}. It is found 
that the nonlinear dynamics of functions in our family is little
more chaotic than the nonlinear dynamics of the family of functions
considered in~\cite{devkeen1}, probably due to having more
bifurcation points, although several qualitative features in the
dynamics of functions in the two families are similar. Further, it
is observed in our investigations that overall visual effect of
chaos is milder in our family than in the family of functions
in~\cite{devkeen1}.

\section{Definitions and Basic Results}
\label{fpnature}
 Let
$\mathbb{C}$ and $\mathbb{\hat{C}}$ denote the complex plane and the extended complex
plane respectively. A point $w$ is said to be a critical point of
$f$ if $f'(w)=0$. The value $f(w_{0})$ corresponding to a critical point
$w_{0}$ is called a critical value of $f$. A point $\alpha\in \mathbb{\hat{C}}$ is
said to be an asymptotic value for $f(z)$, if there is a continuous
curve $\gamma (t)$ satisfying $\lim_{t\rightarrow \infty} \gamma
(t)=\infty$ and $\lim_{t\rightarrow \infty} f(\gamma (t))=\alpha$. A singular value of $f$ is either a critical value or an asymptotic value of $f$. A
function is said to be critically finite if it has only finitely
many singular values. The basic results concerning singular values and their applications in determining the global dynamics of a function are found in~\cite{morosawa}. The Lyapunov exponent is an important tool to measure the chaos. It is known ~\cite{hilborn} that, for $ith$ iterate $x_{i},$ a system's behaviour is chaotic if Lyapunov exponent
\begin{equation*}
L(f)=\lim_{k\rightarrow \infty}\frac{1}{k}\sum_{i=0}^{k-1} \ln |f^{\prime}(x_{i})|
\end{equation*}
 of the function $f(x)$ is a positive number. 

\par
Let
\begin{equation*}
{\cal M}=\left\{ \zeta_{\lambda}(z)=\lambda\frac{z}{z+1} e^{-z} : \lambda >0, z \in {\mathbb{C}}\right\}
\end{equation*}
be one parameter family of transcendental meromorphic functions.
It is easily seen that
the functions $\zeta \in {\cal M}$ have two critical values
$\frac{3-\sqrt{5}}{2}\lambda e^\frac{1-\sqrt{5}}{2}$,
$\frac{3+\sqrt{5}}{2}\lambda e^\frac{1+\sqrt{5}}{2}$, one finite
asymptotic value 0 and have rational Schwarzian derivative
$SD(\zeta)=\displaystyle{-\frac{z^4 + 2 z^3 - 3 z^2 - 4 z+18}{2 (z^2 + z
-1)^2}}$. 
\par
The nature of fixed points (i.e. the points $x$ satisfying $\zeta(x)=x$) and periodic points (i.e. the points $x$ satisfying $\zeta^n (x) = x$, $n=2,3,\dots$) of the functions $\zeta\in {\cal M}$, needed in the sequel, are found in the present section.
\par
The number and locations of fixed points of the function $\zeta_{\lambda}(x)$ for $\lambda
>0$ on real axis are described as follows:

Let $\zeta_{\lambda} \in {\cal M}$. Then, besides the fixed point
$x=0$, the locations of real fixed points of the
function $\zeta_{\lambda}(x)=\lambda \frac{x e^{-x}}{x+1}$ are given
by the following: {\it (a)} For $0<\lambda < 1$,
$\zeta_{\lambda}(x)$ has exactly one fixed point and that is in the interval
$(-1, 0)$. {\it (b)} For $\lambda = 1$, $\zeta_{\lambda}(x)$ has no
non-zero fixed point. {\it (c)} For $1< \lambda <\lambda^{*}$,
$\zeta_{\lambda}(x)$ has exactly one fixed point and that is in the interval
$(0, \infty)$, where $\lambda^{*}= (\sqrt{2}+1) e^{\sqrt{2}}$. {\it (d)} For $\lambda \geq\lambda^{*}$,
$\zeta_{\lambda}(x)$ has exactly one fixed point. 

\par
For $\lambda>\lambda^{*}$, the function  $\zeta_{\lambda}(x)$
has periodic points of period greater than or equal to $2$ in $(0, \infty)$ in addition to having one fixed
point. These periodic points are roots of $\displaystyle
\zeta^{n}_{\lambda}(x) \equiv \lambda \frac{\zeta^{n-1}_{\lambda}(x)
e^{-\zeta^{n-1}_{\lambda}(x)}}{\zeta^{n-1}_{\lambda}(x)+1} = x$. In
the case $n=2$, the periodic points $p_{11}$ and $p_{12}$ of $\zeta_{12}(x)$ are computationally obtained as
$p_{11}\approx 0.748218$, $p_{12}\approx 2.43034$ ({\it
Fig.~\ref{fpfigs}}) for $\lambda>\lambda^{*}$.
\begin{figure}[h]
\begin{center}
\includegraphics[width=5.2cm,height=4.8cm]{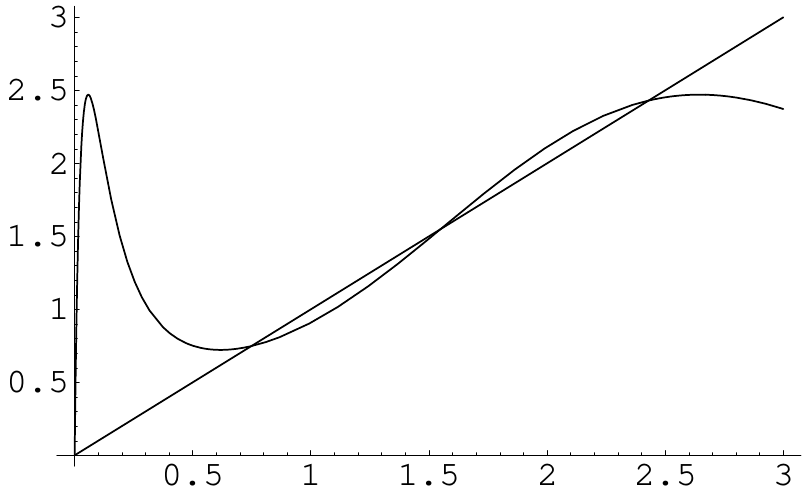} \\
\caption{Periodic Points of period 2 of $\zeta_{\lambda}(x)$ for $\lambda=12$}
\label{fpfigs}
\end{center}
\end{figure} 

\par
The computed value of the fixed point of $\zeta_{\lambda}(x)$ is $-0.314923$ if $\lambda=0.5$ $(0<\lambda<1)$ and this value is $0.374823$ if $\lambda=2$ $(1<\lambda<\lambda^{*})$.

\par
The nature of fixed points of
$\zeta_{\lambda}(x)$ for different values of parameter $\lambda$ is
described in the following theorem:
\begin{thm} \label{thm102}
Let $\zeta_{\lambda} \in {\cal M}$. Then, {\it (i)} If
$0<\lambda<1$, then the fixed point $r_{\lambda} \in (-1, 0)$ is
repelling and the fixed point $0$ is attracting. {\it (ii)} If
$\lambda=1$, then the fixed point $0$ is rationally indifferent .
{\it (iii)} If $1<\lambda<\lambda^{*}$, then the fixed point
$a_{\lambda} \in (0, \sqrt{2})$ is attracting and the fixed point
$0$ is repelling. {\it (iv)} If $\lambda=\lambda^{*}$, then the
fixed point $\sqrt{2}$ is rationally indifferent and the fixed point
$0$ is repelling. {\it (v)} If $\lambda>\lambda^{*}$, then the
 fixed point $0$ is repelling.
\end{thm}
\noindent {\it Proof.} Since $\zeta^{\prime}_{\lambda}(x)=-\lambda
\frac{x^2 + x -1}{(x+1)^{2}} e^{-x}$ and the non-zero fixed points
of $\zeta_{\lambda}(x)$ are solutions of the equation $(x+1)
e^{x}=\lambda$, the multiplier $\zeta^{\prime}_{\lambda}(x_{f})$ of
the fixed point $x_{f}$ is given by
\begin{equation}
|\zeta^{\prime}_{\lambda}(x_{f})|= \frac{|x_{f}^{2}+ x_{f} -1|}{(x_{f}+1)}
\label{df22}
\end{equation}
\par
Let $p(x) =|x^2+x-1| - (x+1)$. It is easily seen that $p(x)$ is a
continuous function and $p(x)<0$ for $x\in (0, \sqrt{2})$, $p(x)>0$ for $x
\in (-1, 0)\cup (\sqrt{2}, \infty)$. Therefore, from Equation~(\ref{df22}), it follows that
\begin{align}
|\zeta^{\prime}_{\lambda}(x_{f})|< 1&\quad\text{for}\;\; x_{f}\in (0,
\sqrt{2}) \label{fd11} \\
|\zeta^{\prime}_{\lambda}(x_{f})|= 1&\quad\text{for}\;\; x_{f}=\sqrt{2} \label{fd12} \\
|\zeta^{\prime}_{\lambda}(x_{f})|> 1&\quad\text{for}\;\; x_{f}\in (-1, 0)
\cup \left(\sqrt{2},\infty\right).  \label{fd13}
\end{align}

Using the Inequalities (\ref{fd11})-(\ref{fd13}) of the derivative of
the function $\zeta_{\lambda}(x)$ at its nonzero fixed points, the behaviour
of fixed points of $\zeta_{\lambda}(x)$, for various values of
$\lambda$ is as follows:
\begin{enumerate}[(i)]
\item
For fixed point $r_{\lambda} \in (-1, 0)$, by Inequality
(\ref{fd13}), $|\zeta^{\prime}_{\lambda}(r_{\lambda})|> 1$.
Therefore, $r_{\lambda}$ is a repelling fixed point of
$\zeta_{\lambda}(x)$. Further, since, $0 <
\zeta^{\prime}_{\lambda}(0)=\lambda < 1$, the point $0$ is an
attracting fixed point of $\zeta_{\lambda}(x)$ for $0<\lambda<1$.
\item
Since, $\zeta^{\prime}_{\lambda}(0) =1$, the point $0$ is a
rationally indifferent fixed point of $\zeta_{\lambda}(x)$ for
$\lambda =1$.

\item
For fixed point $a_{\lambda} \in (0, \sqrt{2})$, by Inequality
(\ref{fd11}), $|\zeta^{\prime}_{\lambda}(a_{\lambda})| < 1 $ so that
 $a_{\lambda}$ is an attracting fixed point of $\zeta_{\lambda}(x)$.
Further, since $\zeta^{\prime}_{\lambda}(0)=\lambda > 1$, the point
$0$ is a repelling fixed point of $\zeta_{\lambda}(x)$ for
$1<\lambda<\lambda^{*}$.

\item
From Equation (\ref{fd12}), we get
$|\zeta^{\prime}_{\lambda}(\sqrt{2})| =1$. Therefore, $\sqrt{2}$ is
a rationally indifferent fixed point of $\zeta_{\lambda}(x)$.
Further, since $|\zeta^{\prime}_{\lambda}(0)|=\lambda^{*} > 1$, the
point $x=0$ is a repelling fixed point of $\zeta_{\lambda}(x)$ for
$\lambda =\lambda^{*}$.

\item Since, $|\zeta^{\prime}_{\lambda}(0)|=\lambda > \lambda^{*} > 1$, it
gives that $0$ is a repelling fixed point of $\zeta_{\lambda}(x)$ for
$\lambda>\lambda^{*}$. \hfill $\square$
\end{enumerate}

\begin{rem} For $\lambda>\lambda^{*}$, the
periodic cycle of the function
$\zeta_{\lambda}(x)$ of period greater than or equal to 2 may be
attracting, repelling or indifferent. For $\lambda = 12 >\lambda^{*}$, for instance, it
is found that $\zeta_{\lambda}(x)$ has 2-cycle periodic points
$p_{11}\approx 0.748218$ and $p_{12}\approx 2.43034$, so that
$\zeta^{\prime}_{\lambda}(p_{11}) = -0.57235$,
$\zeta^{\prime}_{\lambda}(p_{12}) = -0.65847$. It follows that
$|\zeta^{\prime}_{\lambda}(p_{11})
\zeta^{\prime}_{\lambda}(p_{12}) | = 0.376878 < 1$ for $\lambda =
12$. Consequently, the periodic 2-cycle of $\zeta_{12}(x)$ is
attracting. Similarly, for $\lambda=18.5>\lambda^{*}$,
$\zeta_{\lambda}(x)$ has 2-cycle periodic points $q_{11}\approx
0.408442$ and $q_{12}\approx 3.56598$, so that
$\zeta^{\prime}_{\lambda}(q_{11}) = 2.63285$,
$\zeta^{\prime}_{\lambda}(q_{12}) = -0.383357$ and consequently,
$|\zeta^{\prime}_{\lambda}(q_{11})
\zeta^{\prime}_{\lambda}(q_{12}) | = 1.00932 > 1$ for
$\lambda=18.5$. It therefore follows that the periodic 2-cycle of
$\zeta_{18.5}(x)$ is repelling. While,  for $\lambda=18.44505
>\lambda^{*}$, $\zeta_{\lambda}(x)$ has 2-cycle periodic
points $s_{11}\approx 0.409865$ and $s_{12}\approx 3.55911$, so that
$\zeta^{\prime}_{\lambda}(s_{11}) = 2.60007$,
$\zeta^{\prime}_{\lambda}(s_{12}) = -0.384606$ and consequently,
$|\zeta^{\prime}_{\lambda}(s_{11})
\zeta^{\prime}_{\lambda}(s_{12})|=1$ for $\lambda=18.44505$.
Therefore, the periodic 2-cycle of $\zeta_{18.44505}(x)$ is
indifferent. Thus, for $\lambda>\lambda^{*}$, the periodic cycle of
period greater than or equal to 2 may be attracting, repelling or
indifferent.
\end{rem}

\section{Bifurcation in Dynamics of $\zeta_{\lambda}\in {\cal M}$ }     \label{bifur3} In this
section, the bifurcation in the dynamics of the functions $\zeta_{\lambda}(x)$,
$x \in {\mathbb{R}}\backslash\{-1\}$, is investigated. The following theorem gives that the bifurcations in the dynamics of
$\zeta_{\lambda}(x)$, $x \in {\mathbb{R}}\backslash T_{p}$, occur at
$\lambda  = 1$ and $\lambda = \lambda^{*}$, where $T_{p}$ is the set
of the points that are in the backward orbits of the pole $-1$ of
$\zeta_{\lambda}(x)$: 

\begin{thm} \label{thm33}
Let $\zeta_{\lambda} \in {\cal M}$. Then, {\it (a)} If $0<\lambda<
1, \; \zeta^{n}_{\lambda}(x)\rightarrow 0 \; {\mbox{as}} \;
n\rightarrow \infty$  for $x \in [(-\infty, -1) \cup (-1,
r_{\lambda})\cup (r_{\lambda}, \infty)]\backslash T_{p}$. {\it (b)}
If $\lambda=1, \; \zeta^{n}_{\lambda}(x)\rightarrow 0 \; {\mbox{as}}
\; n\rightarrow \infty$  for $x\in [(-\infty, -1) \cup (-1,
\infty)]\backslash T_{p}$. {\it (c)}  If $1<\lambda< \lambda^{*}, \;
\zeta^{n}_{\lambda}(x)\rightarrow a_{\lambda} \; {\mbox{as}} \;
n\rightarrow \infty$  for $x \in [(-\infty, -1) \cup (-1, 0) \cup
(0, \infty)]\backslash T_{p}$. {\it (d)} If $\lambda=\lambda^{*}, \;
\zeta^{n}_{\lambda}(x)\rightarrow \sqrt{2} \; {\mbox{as}} \;
n\rightarrow \infty$  for $x\in [(-\infty, -1) \cup (-1, 0) \cup (0,
\infty)]\backslash T_{p}$. {\it (e)}  If $\lambda > \lambda^{*}$,
the orbits $\{\zeta^{n}_{\lambda}(x)\}$ repel for all $x \in {\mathbb{R}}\backslash T_{p}.$
\end{thm}
\noindent {\it Proof.} Let $t_{\lambda}(x)=\zeta_{\lambda}(x)-x$ for
$ x \in {\mathbb{R}} \backslash\{-1\}$. It is easily seen that
$t_{\lambda}(x)$ is continuously differentiable for  $ x\in
{\mathbb{R}}\backslash \{-1\}$. Observe that the fixed points of
$\zeta_{\lambda}(x)$ are zeros of $t_{\lambda}(x)$. Using the
function $t_{\lambda}(x)$, the dynamics of $\zeta_{\lambda}(x)$, for
various values of $\lambda$, is now described as follows:
\begin{enumerate}
\item[(a)] If $0<\lambda<1$, it follows by Theorem~\ref{thm102} that
$\zeta_{\lambda}(x)$ has an attracting fixed point $0$ and a repelling
fixed point $r_{\lambda}$. Since $t^{\prime}_{\lambda}(r_{\lambda})>0$
and in a neighbourhood of $r_{\lambda}$ the function $t^{\prime}_{\lambda}(x)$
is continuous, $t^{\prime}_{\lambda}(x) > 0$
in some neighbourhood of $r_{\lambda}$. Therefore, $t_{\lambda}(x)$ is
increasing in a neighbourhood of $r_{\lambda}$. By the continuity of
$t_{\lambda}(x)$, for sufficiently small $\delta_{1} >0$,
$t_{\lambda}(x)> 0$ in $(r_{\lambda}, r_{\lambda} + \delta_{1})$ and
$t_{\lambda}(x) < 0$ in $(r_{\lambda}-\delta_{1}, r_{\lambda})$.
Further, since $t^{\prime}_{\lambda}(0)<0$ and $t^{\prime}_{\lambda}(x)$
is continuous in a neighbourhood of $0$, $t^{\prime}_{\lambda}(x) < 0$
in some neighbourhood of $0$, $t_{\lambda}(x)$ is decreasing in
a neighbourhood of $0$. By the continuity of $t_{\lambda}(x)$, for
sufficiently small $\delta_{2} >0$, $t_{\lambda}(x) > 0$ in
$(-\delta_{2}, 0)$ and $t_{\lambda}(x) < 0$ in $(0, \delta_{2})$.
Since $t_{\lambda}(x) \neq 0$ in $(r_{\lambda}, 0)$,
\begin{equation}
t_{\lambda}(x) = \zeta_{\lambda}(x)-x \left\{ \begin{array}{lcl}
>0 & & \mbox{for} \; x \in (r_{\lambda}, 0) \\
<0 & & \mbox{for} \; x \in (-1, r_{\lambda}) \cup (0, \infty).
\end{array} \right.
\label{txx41}
\end{equation}

Therefore, for $0<\lambda<1$, the dynamics of $\zeta_{\lambda}(x)$ is as
follows: (i) $\,$ By~(\ref{txx41}), it follows that, for
$x\in (0,\infty)$, $\zeta_{\lambda}(x) <x$. Since
$\zeta_{\lambda}(x) > 0$ for $ x\in (0, \infty)$, the sequence $\{\zeta^{n}_{\lambda}(x)\}$
is decreasing and bounded below by $0$. Hence
$\zeta^{n}_{\lambda}(x) \rightarrow 0$ as $n \rightarrow \infty$ for
$x\in (0, \infty)$. (ii) $\,$ By~(\ref{txx41}), it follows that, for
$x\in (r_{\lambda}, 0)$, $\zeta_{\lambda}(x) > x$. Since
$\zeta_{\lambda}(x)$ is negative and increasing for $x\in
(r_{\lambda}, 0)$, the sequence $\{\zeta^{n}_{\lambda}(x)\}$ is
increasing and bounded above by $0$. Hence $\zeta^{n}_{\lambda}(x)
\rightarrow 0$ as $n \rightarrow \infty$ for $x\in (r_{\lambda},
0)$. (iii) Since $\zeta_{\lambda}(x)$ maps the interval $(-\infty,
-1)$  into $(0, \infty)$, by Case(i), we get $\zeta^{n}_{\lambda}(x)
\rightarrow 0$ as $n \rightarrow \infty$ for $x\in (-\infty, -1)$.
(iv) The forward orbit of $\zeta_{\lambda}(x)$ for each point in
$(-1, r_{\lambda}) \backslash T_{p}$ is contained in $(-\infty,
-1)$. Thus, by Case(iii), $\zeta^{n}_{\lambda}(x) \rightarrow 0$ as
$n \rightarrow \infty$ for $x\in (-1, r_{\lambda})\backslash T_{p}$.
Figs.~\ref{webdiagram1}(i) and (ii) show the web diagrams of
dynamics of $\zeta_{\lambda}(x)$, $0 <\lambda <1$.
\begin{figure}[h]
\begin{minipage}{0.46\textwidth}
\includegraphics[angle=-90, scale=0.38]{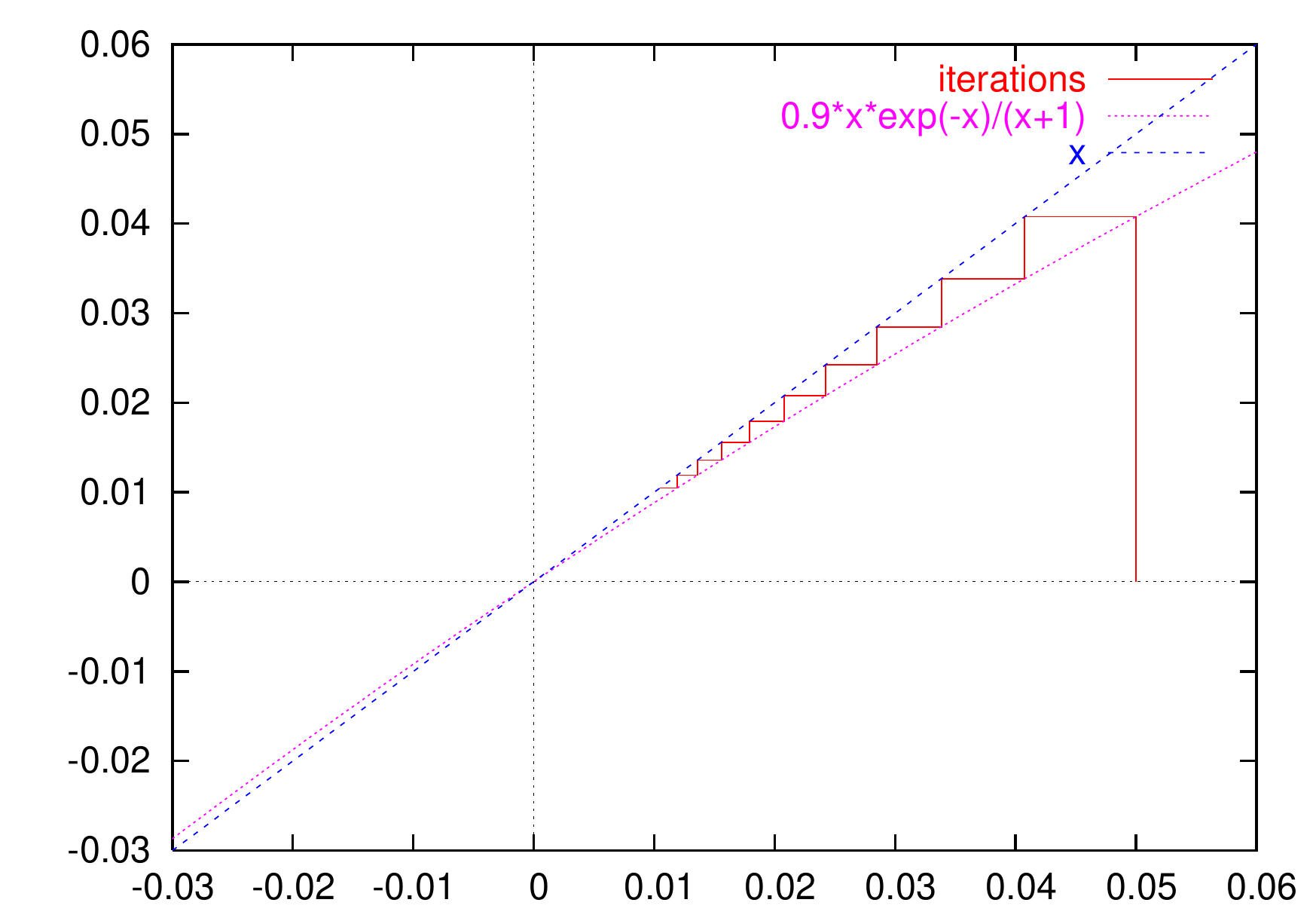}\\\\
\centering{(i) $n=10, \; x_{0}=0.05$}
\end{minipage} \hfill
\begin{minipage}{0.46\textwidth}
\includegraphics[angle=-90, scale=0.38]{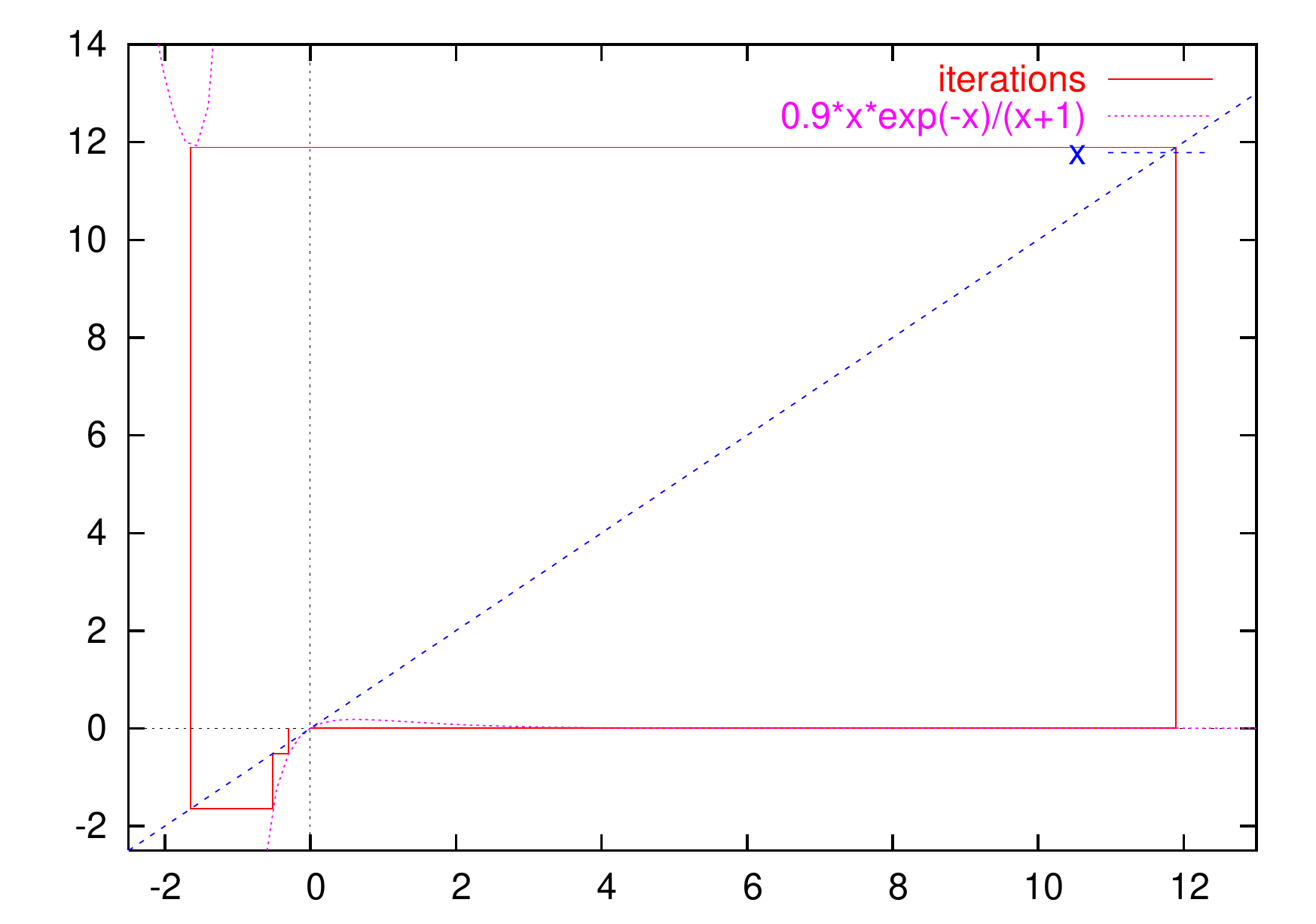}\\\\
\centering{(ii) $n=10, \; x_{0}=-0.3 $}
\end{minipage}
\caption{Web diagrams of $\zeta_{\lambda}(x)$ for $\lambda=0.9$} \label{webdiagram1}
\end{figure}
\item[(b)] If $\lambda=1$, by Theorem~\ref{thm102}, $\zeta_{\lambda}(x)$
has a rationally indifferent fixed point $0$. This implies
$t^{\prime}_{\lambda}(0) = 0$ and $t^{\prime\prime}_{\lambda}(0) < 0$,
so that $t_{\lambda}(x)$ has maxima at $0$. But $t_{\lambda}(0) = 0$
therefore $t_{\lambda}(x) <0$ in a neighbourhood of $0$. By continuity
of $t_{\lambda}(x)$, for $\delta >0$, $t_{\lambda}(x) < 0$ in
$(-\delta, 0) \cup (0, \delta)$. Since $t_{\lambda}(x) \neq 0$ in
$(-1, 0) \cup (0, \infty)$, it now follows that
\begin{equation}
t_{\lambda}(x) = \zeta_{\lambda}(x)-x <0  \;\; \mbox{for} \; x \in (-1, 0) \cup (0, \infty).
\label{txx42}
\end{equation}
Therefore, for $\lambda=1$, the dynamics of  $\zeta_{\lambda}(x)$ is found to be as follows: (i) $\,$ By~(\ref{txx42}), it
follows that, for $x\in (0, \infty)$, $\zeta_{\lambda}(x) < x$.
Since $\zeta_{\lambda}(x)>0$ for $ x\in (0, \infty)$, the sequence
$\{\zeta^{n}_{\lambda}(x)\}$ is decreasing and bounded below by $0$.
Hence $\zeta^{n}_{\lambda}(x) \rightarrow 0$ as $n \rightarrow
\infty$ for $x\in (0, \infty)$. (ii) By arguments similar to those
used in the proof for Case(iii) of (a), it follows that
$\zeta^{n}_{\lambda}(x) \rightarrow 0$ as $n \rightarrow \infty$ for
$x\in (-\infty, -1)$. (iii) The forward orbit of
$\zeta_{\lambda}(x)$ for each point in $(-1, 0) \backslash T_{p} $
is contained in the interval $(-\infty, -1)$. Therefore, by
Case(ii), $\zeta^{n}_{\lambda}(x) \rightarrow 0$ as $n \rightarrow
\infty$. The web diagrams of dynamics of $\zeta_{\lambda}(x)$,
$\lambda = 1$, are given in Figs.~\ref{webdiagram2}(i) and (ii).
\begin{figure}[h]
\begin{minipage}{0.46\textwidth}
\includegraphics[angle=-90, scale=0.38]{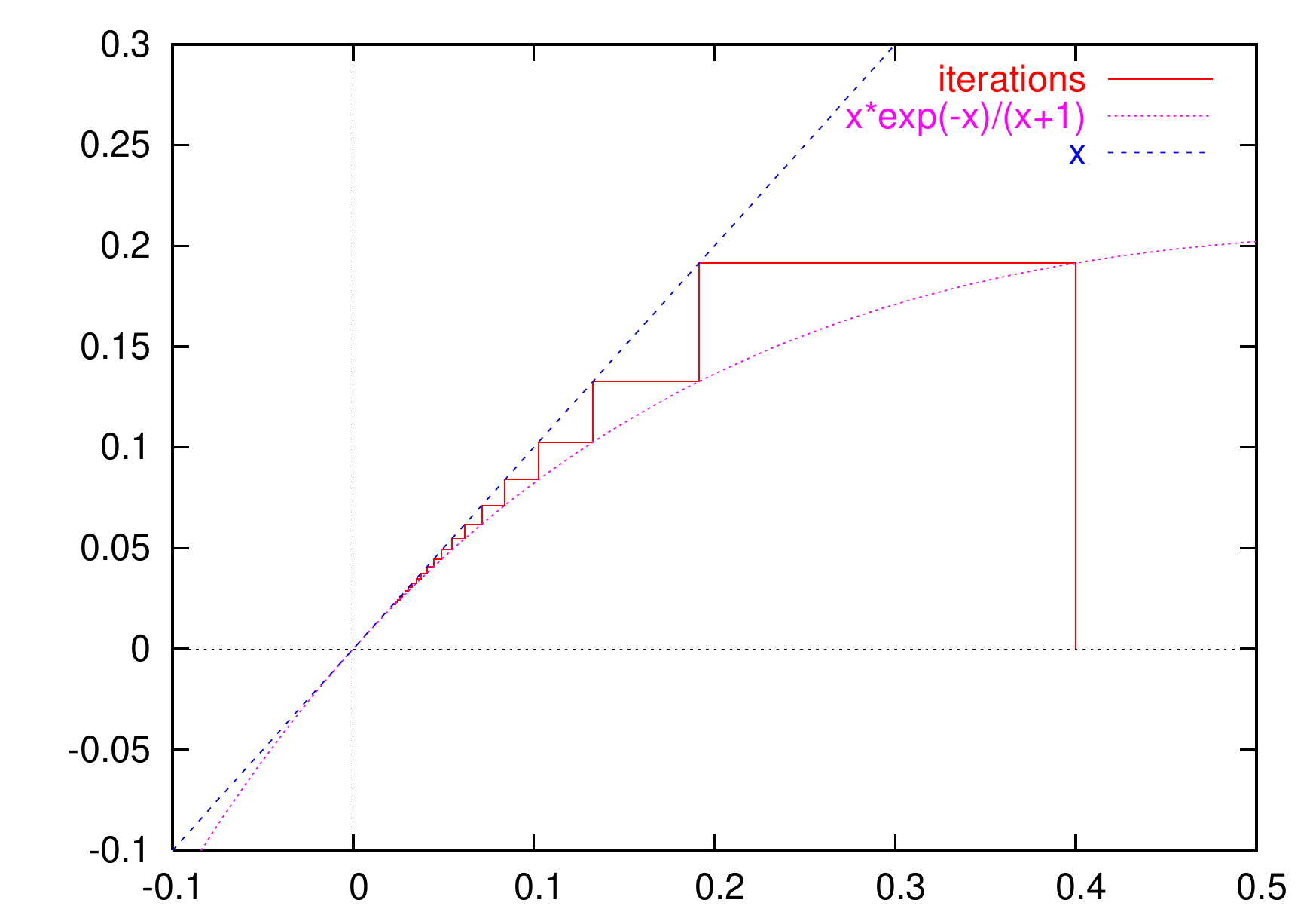}\\\\
\centering{(i) $n=20, \; x_{0}=0.4$}
\end{minipage}  \hfill
\begin{minipage}{0.46\textwidth}
\includegraphics[angle=-90, scale=0.38]{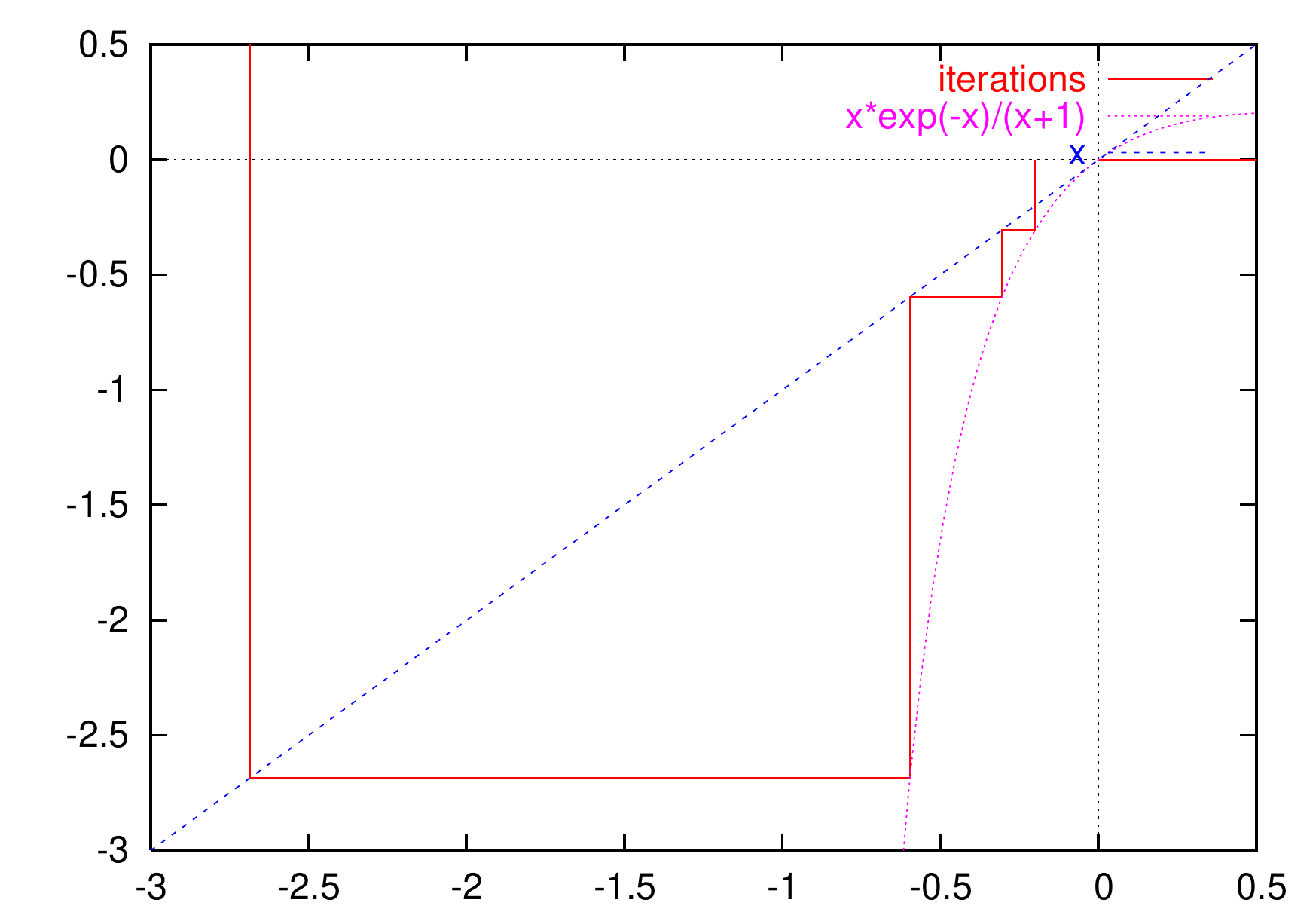}\\\\
\centering{(ii) $n=20, \; x_{0}=-0.2 $}
\end{minipage}
\caption{Web diagrams of $\zeta_{\lambda}(x)$ for $\lambda=1$} \label{webdiagram2}
\end{figure}

\item[(c)] If $1<\lambda <\lambda^{*}$, by Theorem~\ref{thm102},
$\zeta_{\lambda}(x)$ has an attracting fixed point $a_{\lambda}$ and
a repelling fixed point $0$. The dynamics of $\zeta_{\lambda}(x)$, for $1<\lambda <\lambda^{*}$,  is found to be as follows: (i) First let $x \in
(\frac{-1+\sqrt{5}}{2}, a_{\lambda})$. Since
$\zeta^{\prime\prime}_{\lambda}(x) < 0$ for $x \in (-\infty,
x_{0})$, $\zeta^{\prime}_{\lambda}(x)$ is decreasing for $x \in
(-\infty, x_{0})$, where $x_{0} (\approx 1.26953)$ is a solution of
the equation $x^3+ 2 x^2 -x -4 = 0$. Since
$\zeta^{\prime}_{\lambda}(\frac{-1+\sqrt{5}}{2})=0$ and $-1 <
\zeta^{\prime}_{\lambda}(a_{\lambda}) < 0$, by Mean Value Theorem
$|\zeta_{\lambda}(x) - a_{\lambda}| < |x - a_{\lambda}|$. Therefore,
$\zeta^{n}_{\lambda}(x) \rightarrow a_{\lambda}$ as $n \rightarrow
\infty$ for $x \in (\frac{-1+\sqrt{5}}{2}, a_{\lambda})$. Next, for
each $x\in (0, \; \frac{-1+\sqrt{5}}{2})$, there exits $n_{0} \in
\mathbb{N}$ such that $\zeta^{n_{0}}_{\lambda}(x) \; \in \;
(\frac{-1+\sqrt{5}}{2}, \; a_{\lambda})$. It follows that
$\zeta^{n}_{\lambda}(x) \rightarrow a_{\lambda}$ as $ n\rightarrow
\infty$ for $x \in (0, \; \frac{-1+\sqrt{5}}{2})$. Further, for
$(a_{\lambda}, \; \infty)$, $\zeta_{\lambda}(x)$ maps $(a_{\lambda},
\; \infty)$ into $(0, \; \frac{-1+\sqrt{5}}{2})$,
$\zeta^{n}_{\lambda}(x) \rightarrow a_{\lambda}$ as $n \rightarrow
\infty$ for $x \in (a_{\lambda}, \; \infty)$. Hence
$\zeta^{n}_{\lambda}(x) \rightarrow a_{\lambda}$ as $n \rightarrow
\infty$ for $x \in (0, \; \infty)$. (ii) By arguments similar to
those used for Case(iii) of (a), it follows that
$\zeta^{n}_{\lambda}(x) \rightarrow a_{\lambda} \; {\mbox{as}} \; n
\rightarrow \infty$. (iii) The forward orbit of $\zeta_{\lambda}(x)$
for each point in $(-1, 0) \backslash T_{p}$ is contained in the
interval $(-\infty, -1)$. Therefore, by Case(ii),
$\zeta^{n}_{\lambda}(x) \rightarrow a_{\lambda}$ as $n \rightarrow
\infty$. Figs.~\ref{webdiagram3}(i) and (ii) give the web diagrams
of dynamics of $\zeta_{\lambda}(x)$ for $1 < \lambda < \lambda^{*}$.
\begin{figure}[h]
\begin{minipage}{0.48\textwidth}
\includegraphics[angle=-90,scale=0.38]{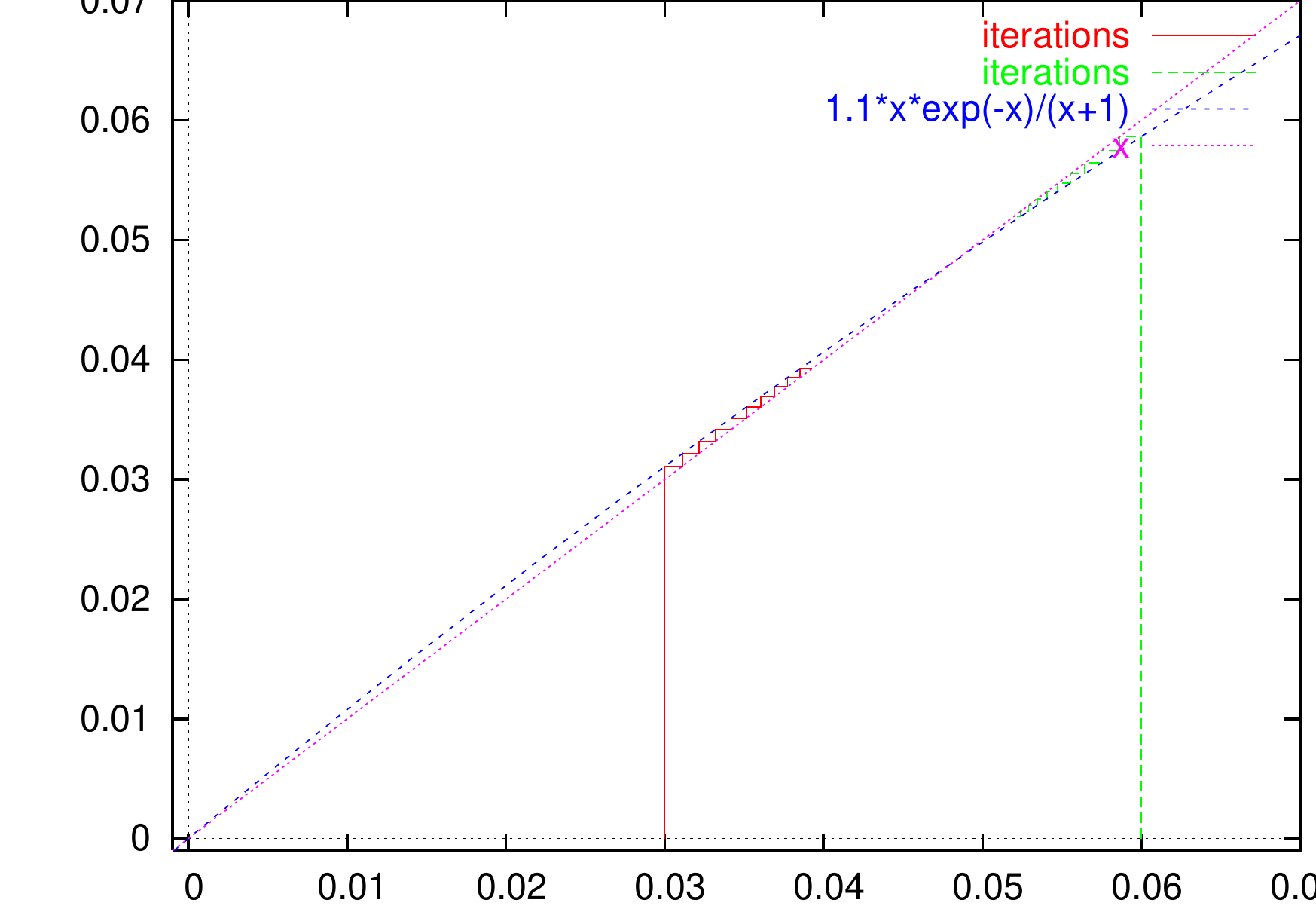}\\\\
\centering{(i) $n=10, x_{0}=0.03, 0.06$}
\end{minipage} \hfill
\begin{minipage}{0.48\textwidth}
\includegraphics[angle=-90,scale=0.38]{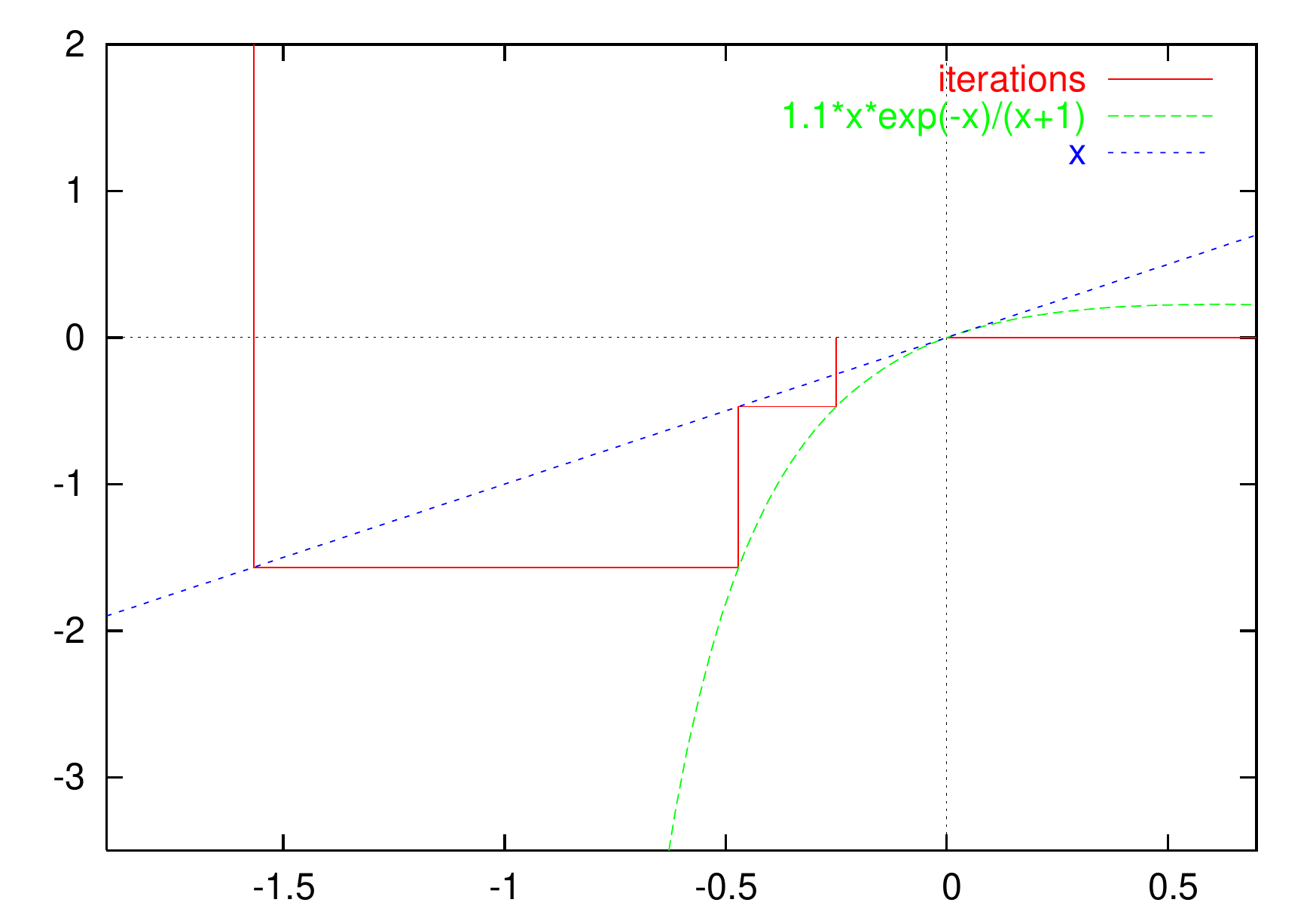}\\\\
\centering{(ii) $n=10, x_{0}=-0.25$}
\end{minipage}
\caption{Web diagrams of $\zeta_{\lambda}(x)$ for $\lambda=1.1$} \label{webdiagram3}
\end{figure}
\item[(d)] If $\lambda =\lambda^{*}$, by Theorem~\ref{thm102},
$\zeta_{\lambda}(x)$ has a rationally indifferent fixed point
$\sqrt{2}$ and a repelling fixed point $0$. The dynamics of
$\zeta_{\lambda}(x)$, for $\lambda =\lambda^{*}$, is as follows: (i) First
let $ x\in (\frac{-1+\sqrt{5}}{2}, \sqrt{2})$, since
$\zeta^{\prime\prime}_{\lambda}(x) < 0$ for $x \in (-\infty,
x_{0})$, $\zeta^{\prime}_{\lambda}(x)$ is decreasing for $x \in
(-\infty, x_{0})$, where $x_{0}(\approx 1.26953)$ is a solution of
the equation $x^3+2x^2-x-4=0$. The rest of proof of assertion is
same as Case(i) of (c). (ii) By arguments similar to those used in
the proof for Case(iii) of (a), it follows that
$\zeta^{n}_{\lambda}(x) \rightarrow \sqrt{2}$ as $n \rightarrow
\infty$. (iii) Since forward orbit of $\zeta_{\lambda}(x)$ for each
point in $(-1, 0) \backslash T_{p}$ is contained in $(-\infty, -1)$,
by Case(ii), $\zeta^{n}_{\lambda}(x) \rightarrow \sqrt{2}$ as
$n\rightarrow \infty$. The web diagram of dynamics of
$\zeta_{\lambda}(x)$, $\lambda = \lambda^{*}$, is given by Fig.~\ref{webdiagram45}(i).
\begin{figure}[h]
\begin{minipage}{0.48\textwidth}
\includegraphics[angle=-90,scale=0.38]{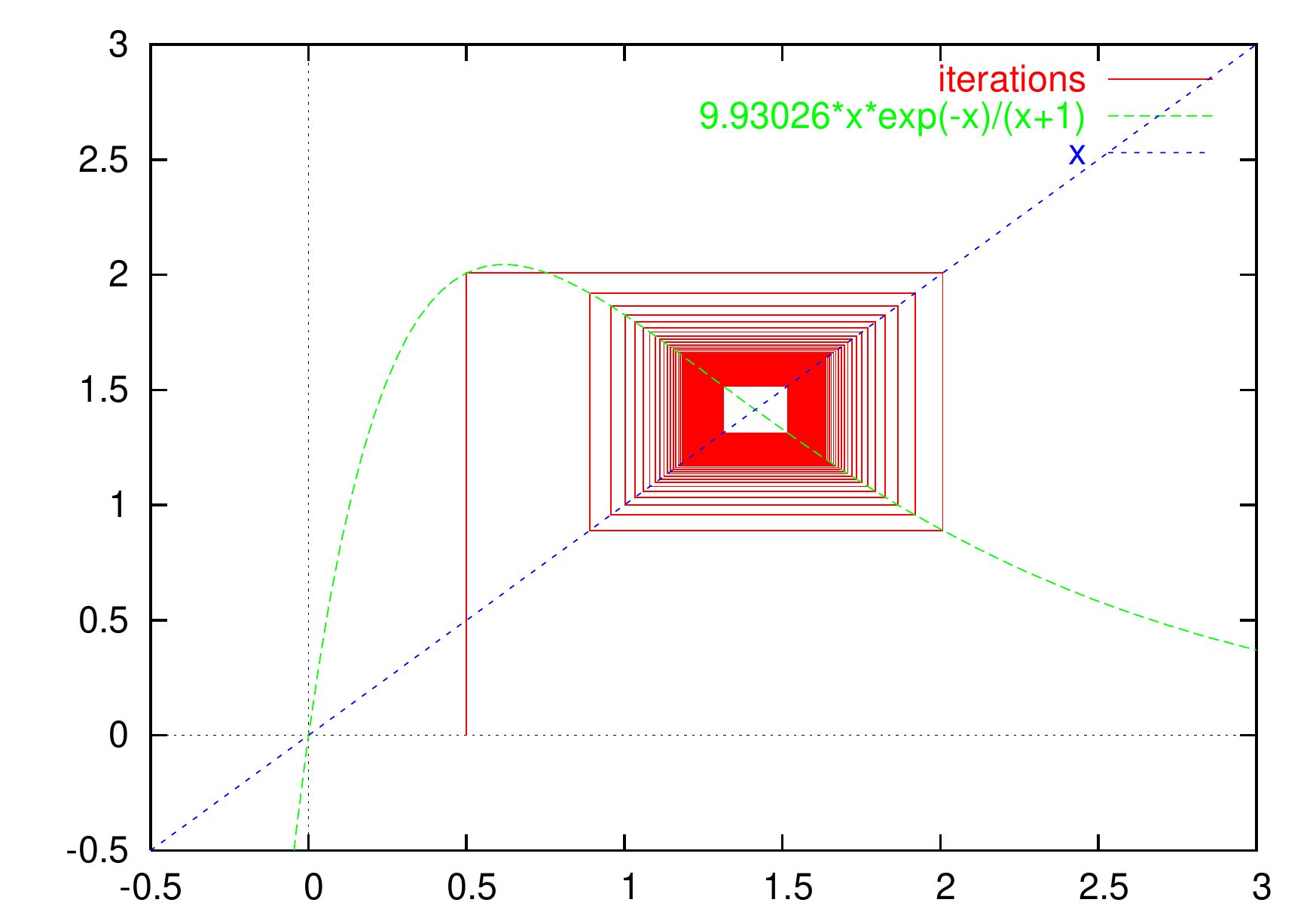}\\\\
\centering{(i) $\lambda=9.93026$, $n=200, x_{0}=0.5$}
\end{minipage} \hfill
\begin{minipage}{0.48\textwidth}
\includegraphics[angle=-90,scale=0.38]{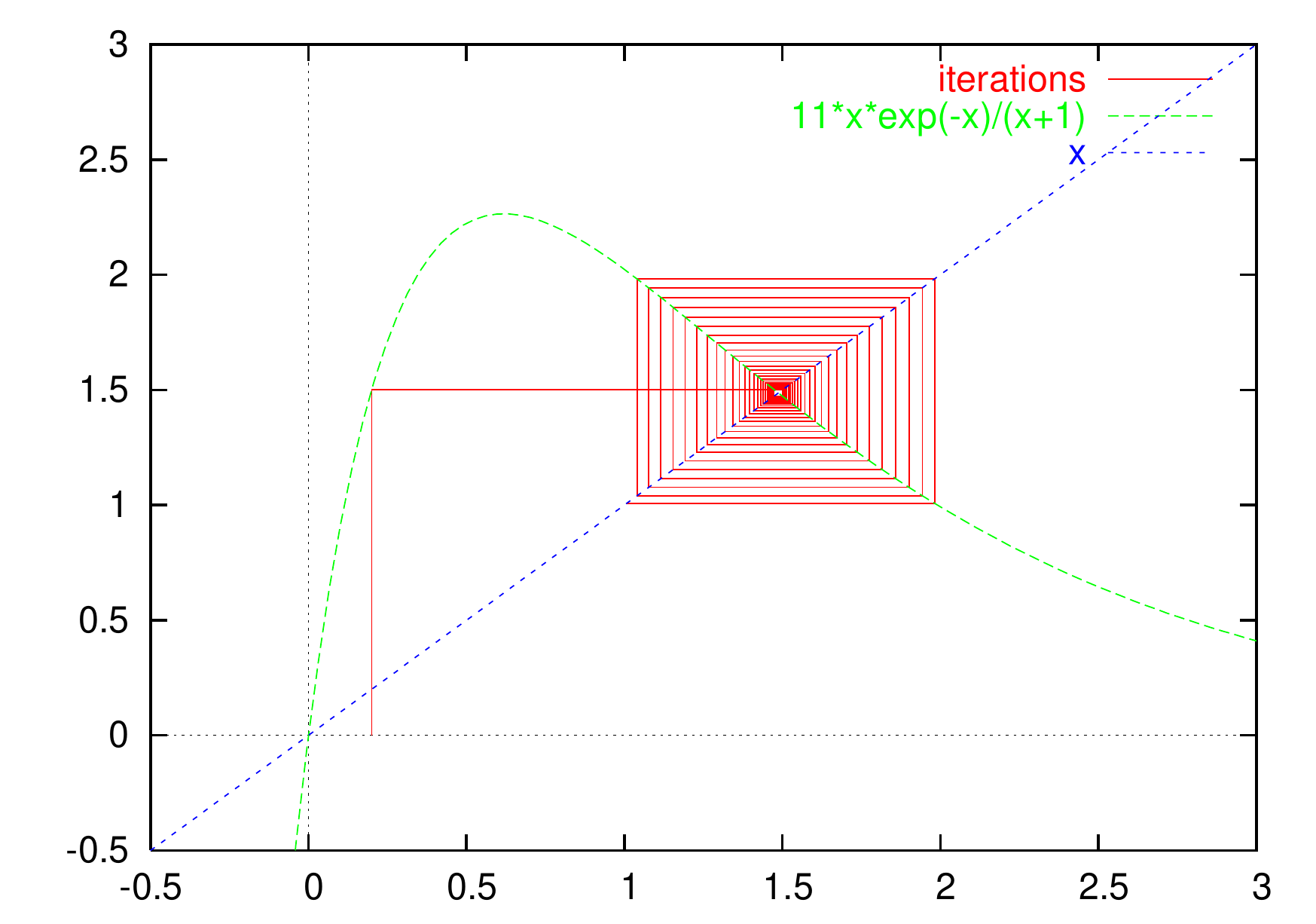}\\\\
\centering{(ii) $\lambda=11$, $n=50, x_{0}=0.2$}
\end{minipage}
\caption{Web diagram of $\zeta_{\lambda}(x)$}
\label{webdiagram45}
\end{figure}
\item[(e)] If $\lambda >\lambda^{*}$, by Theorem~\ref{thm102},
$\zeta_{\lambda}(x)$ has a repelling fixed point $0$. Using arguments similar than above cases, the orbits $\{\zeta^{n}_{\lambda}(x)\}$ repel for all $x \in {\mathbb{R}}\backslash T_{p}.$ . Fig.~\ref{webdiagram45}(ii) gives the web diagram of dynamics of $\zeta_{\lambda}(x)$ for $\lambda > \lambda^{*}$. \hfill $\square$ 	
\end{enumerate}

\begin{rem} \label{rem21}
 If $\lambda >\lambda^{*}$, a periodic
cycle of period greater than or equal to 2 may be attracting,
repelling or indifferent. The dynamics of $\zeta_{\lambda}(x)$ in
this case is now described as follows: (i) Since, for $n=2$,
$\zeta_{\lambda}(x)$ has an attracting, repelling or indifferent
periodic cycle of period 2 and 0 is a repelling fixed point, the
iterations either converges to attracting cycle or keeps moving.
Consequently, all orbits of points in $(0, \; \infty)$ will either
attract, or repel, or keep moving indefinitely. Thus, in this case,
orbits $\{\zeta^{n}_{\lambda}(x)\}$ either attract, or repel, or are
chaotic for $x\in (0, \; \infty)$. (ii) Since $\zeta_{\lambda}(x)$
maps each point in  $(-\infty, -1)$ to a point of the interval $(0,
\infty)$, by Case(i), orbits $\{\zeta^{n}_{\lambda}(x)\}$ either
attract, or repel, or are chaotic for $x \in (-\infty, -1)$. (iii)
The forward orbit of $\zeta_{\lambda}(x)$ for each point in $(-1, 0)
\backslash T_{p}$ is contained in the interval $(-\infty, -1)$.
Therefore, by Case(ii), the orbits $\{\zeta^{n}_{\lambda}(x)\}$
either attract, or repel, or are chaotic for $x\in (-1, 0)\backslash
T_{p}$. Fig.~\ref{webdiagram5r} gives the web diagrams
of dynamics of $\zeta_{\lambda}(x)$ for $\lambda > \lambda^{*}$. 
\end{rem}

\begin{figure}[h]
\begin{center}
\includegraphics[angle=-90, scale=0.38]{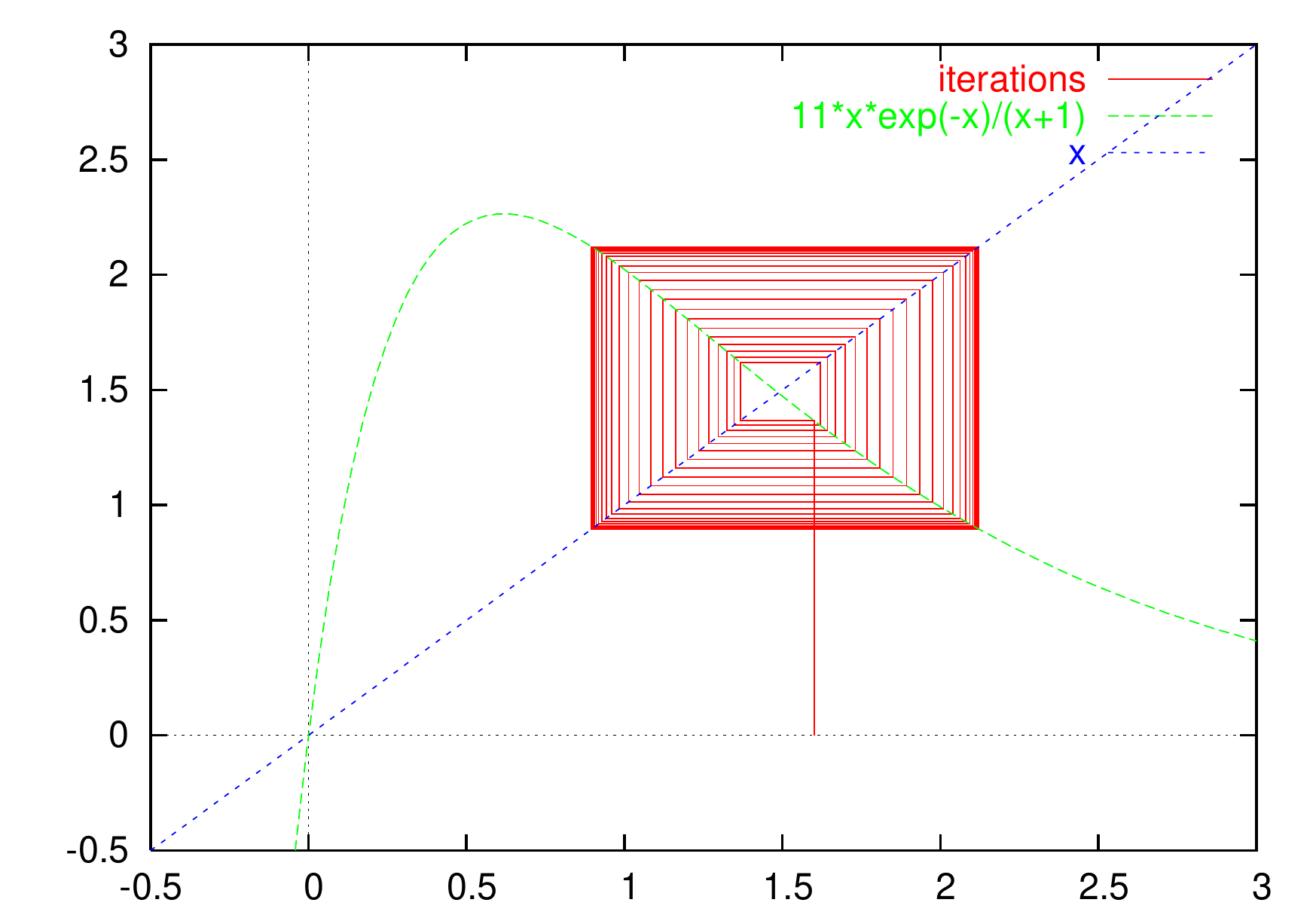}
\end{center}
\caption{Web diagram of $\zeta_{\lambda}(x)$ for $\lambda=11$, $n=50, x_{0}=1.6$} \label{webdiagram5r}
\end{figure}

\begin{rem}
By Theorem~\ref{thm33}, it follows that bifurcations in the dynamics of
the function $\zeta_{\lambda}(x)$, $x \in {\mathbb{R}}\backslash \{-1\}$
occur at several
critical parameter values like $\lambda = 1$ and
$\lambda = (\sqrt{2}+1)e^{\sqrt{2}}$. The bifurcation diagram of the
function $\zeta_{\lambda}(x)$ for
$\lambda > 0$ is given by Fig.~\ref{lambdafix33}.

\begin{figure}[h]
\begin{center}
\includegraphics[height=5.5cm, width=6.8cm]{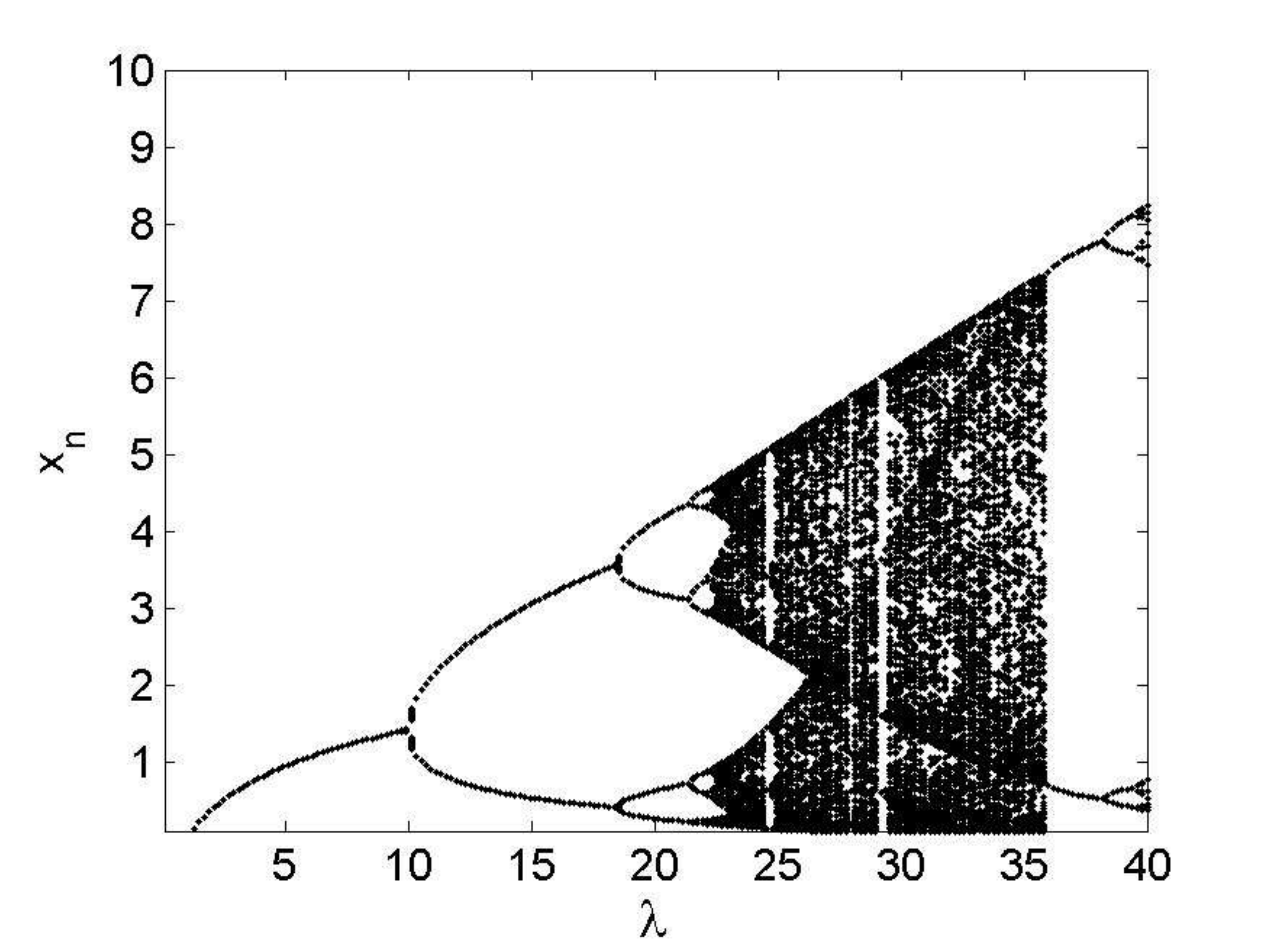}
\caption{Bifurcation diagram of function $\zeta_{\lambda}(x)$}
\label{lambdafix33}
\end{center}
\end{figure}
\end{rem}
In order to quantify chaos in the dynamics, Lyapunov exponents
\begin{equation*}
L(\zeta_{\lambda})=\lim_{k\rightarrow \infty}\frac{1}{k}\sum_{i=0}^{k-1} \ln\bigg(\lambda \frac{|x_{i}^{2}+x_{i}-1| e^{-x_{i}}}{(x_{i}+1)^2}\bigg)
\end{equation*}
of $\zeta_{\lambda}(x)$ for certain values of $\lambda$  are computed for its fixed points and periodic points. From Fig.~\ref{lyapunov}, with $x_{0}$ as a suitable point near a fixed point or periodic point and $k=2000$, shows the computed values of Lyapunov exponents of $\zeta_{\lambda}(x)$ for  $10\le\lambda\le 45$. It is found that, for $\lambda = 25, 42$, the values of  Lyapunov exponent are positive, while, for $\lambda = 2, 12, 18.5, 38$, these are negative. Thus, it is shown that the Lyapunov exponents of $\zeta_{\lambda}(x)$ for certain ranges of the values of the parameter $\lambda$ are positive, demonstrating chaos in the dynamics of the function $\zeta_{\lambda}(x)$ for the values of parameter $\lambda$ in these ranges.

\begin{figure}[h]
\begin{center}
\includegraphics[height=6.0cm, width=9.0cm]{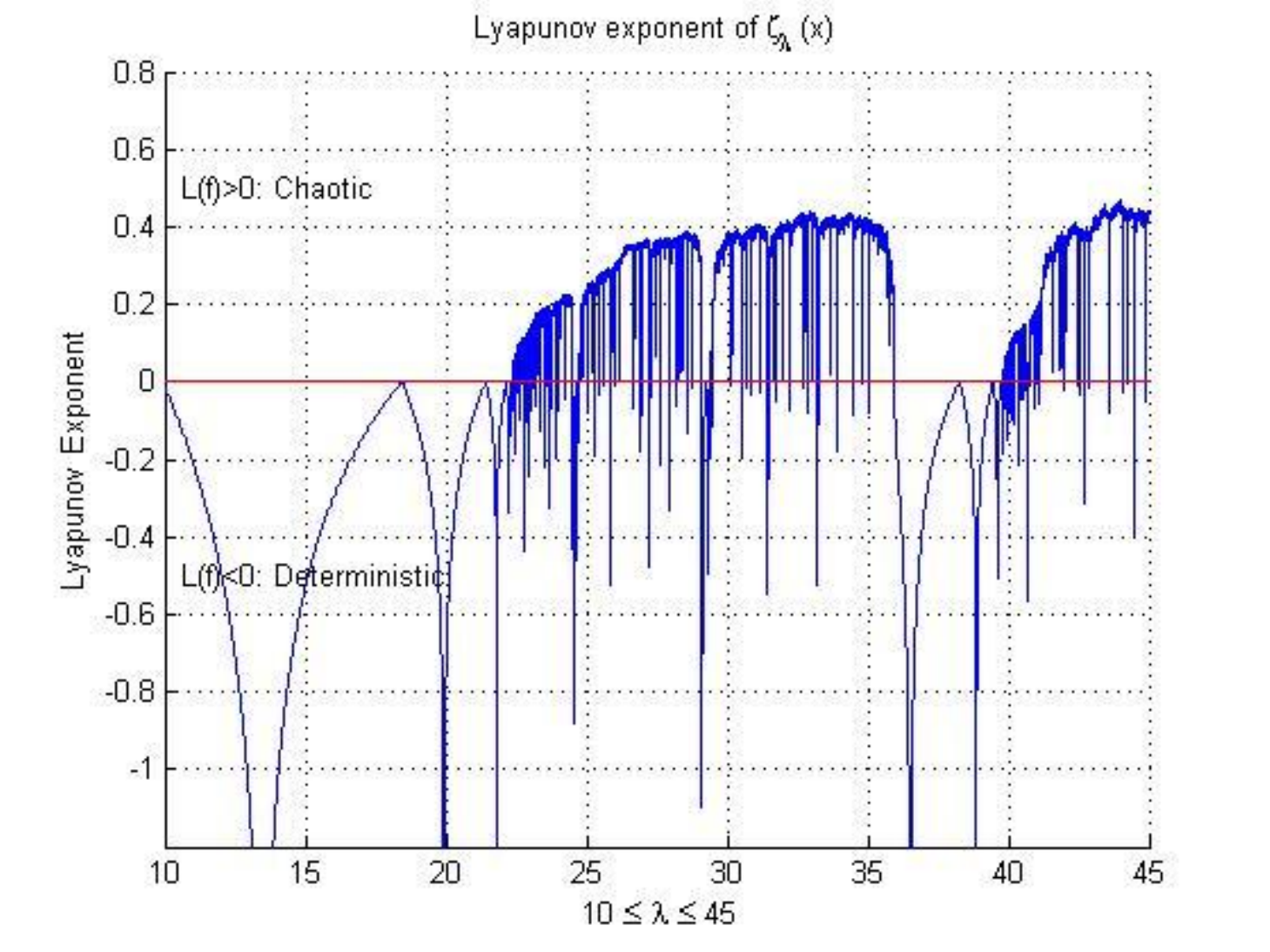}
\caption{Lyapunov Exponents of function $\zeta_{\lambda}(x)$, $10<\lambda<45$}
\label{lyapunov}
\end{center}
\end{figure}

\begin{rem}
Our results on complex dynamics in next
sections are induced from the corresponding results found in the above theorem on real dynamics of the function $\zeta_{\lambda}\in {\cal M}$.
\end{rem}

\section{Complex Dynamics of $\zeta_{\lambda}\in {\cal M}$ and $0<\lambda\leq 1$}\label{sec:a11}
Using Theorem~\ref{thm102}, the complex dynamics of the function
$\zeta_{\lambda}(z)$, $0<\lambda\leq 1$, for $z\in {\mathbb{\hat{C}}}$
 is investigated in the present section. 
\par
If $0<\lambda<1$, by Theorem~\ref{thm102}, the function $\zeta_{\lambda}(z)$ admits the basin of attraction 
${ A(0)}=\{z\in {\mathbb{\hat{C}}} : \zeta^{n}_{\lambda}(z)\rightarrow 0 \;
\text{as} \; n \rightarrow \infty \}$ for its attracting fixed point $0$. The
following theorem gives the characterization of the Julia
set $J(\zeta_{\lambda})$, $0<\lambda< 1$, as the complement of its basin of attraction $A(0)$:
\begin{thm}        \label{thmaa}
Let $\zeta_{\lambda}\in \cal M$, $0<\lambda<1$. Then, 
\begin{equation*}
J(\zeta_{\lambda})={\mathbb{\hat{C}}} \backslash { A(0)}
\end{equation*}
\end{thm}
\begin{proof} Since $z=0$ is an attracting fixed point of $\zeta_{\lambda}(z)$, its
only asymptotic value at $z=0$ lies in the basin of attraction ${ A(0)}$.
The critical values of $\zeta_{\lambda}(z)$ are only finitely many and all of these lie on the
real line. By Theorem~\ref{thm33}(a), the forward orbit of
every critical value tends to attracting fixed point $0$
under iteration. Therefore, all singular
values of $\zeta_{\lambda}(z)$ and their orbits lie in the same
component of ${ A(0)}$. Thus, it follows
that $\zeta_{\lambda}(z)$ has no Siegel disk or Herman ring in
$F(\zeta_{\lambda})$. By
Theorem~\ref{thm102}, $\zeta_{\lambda}(z)$ has only one attracting
fixed point and two repelling fixed points on the real axis so $U$ is
not parabolic domain.
Using Theorem~\ref{thm1114} and Theorem~\ref{thm1115}, the Fatou set $F(\zeta_{\lambda})$ has no wandering
domains and no Baker domains since $\zeta_{\lambda}(z)$ is critically finite meromorphic function.
 Consequently,  the only possible stable component $U$ of
$F(\zeta_{\lambda})$ is the basin of attraction ${ A(0)}$.
\par
Since the Julia set is complement of the Fatou set
$F(\zeta_{\lambda}) = { A(0)}$, the Julia set
$J(\zeta_{\lambda})={\mathbb{\hat{C}}}\backslash {A(0)}$.
\end{proof}

\begin{rem}
By Theorem~\ref{thm33}(a), for $0<\lambda<1$,
$\zeta_{\lambda}^{n}(x) \rightarrow 0$ as $n \rightarrow \infty$ for
$[(-\infty, -1) \cup (-1, r_{\lambda}) \cup (r_{\lambda}, 0) \cup
(0, \infty)]\backslash T_{p}$.  Therefore, it follows that the union of 
intervals ${\mathbb{R}}\backslash \{T_{p}\cup\{r_{\lambda}\}\}$ is
contained in the basin of attraction ${A}(0)$ for $0<\lambda<1$.
\end{rem}

If $\lambda=1$, by Theorem~\ref{thm102}, $\zeta_{\lambda}(z)$ has a real
rationally indifferent fixed point $0$. The following proposition shows
that the Fatou set of $\zeta_{\lambda}(z)$ contains certain intervals
of real line for $\lambda=1$:
\begin{prop}      \label{tpd13}
Let $\zeta_{\lambda}\in \cal M$, $\lambda=1$. Then, $F(\zeta_{\lambda})$ contains the
intervals $(-\infty, -1)$, $(-1, 0)\backslash T_{p}$ and $(0, \infty)$.
\end{prop}
\begin{proof}
By Theorem~\ref{thm33}(b), $\zeta_{\lambda}^{n}(x)
\rightarrow 0$ for $x \in [(-\infty, -1)\cup (-1, 0)\cup (0, \infty)]
\backslash T_{p}$ if $\lambda=1$. Therefore,
the intervals $(-\infty, -1)$, $(-1, 0)\backslash T_{p}$ and $(0,
\infty)$ are contained in the Fatou set $F(\zeta_{\lambda})$.
\end{proof}

\section{Complex Dynamics of $\zeta_{\lambda}\in {\cal M}$, $1 <\lambda \leq\lambda^{*}$ and $\lambda>\lambda^{*}$}\label{sec:c11}
The complex dynamics of the function $\zeta_{\lambda}(z)$ for $z\in
{\mathbb{\hat{C}}}$, $1 < \lambda \leq \lambda^{*}$ and
$\lambda>\lambda^{*}$ is described in this section by using Theorem~\ref{thm102}. 
\par
If $1<\lambda<\lambda^{*}$, by Theorem~\ref{thm102}, the function
$\zeta_{\lambda}(z)$ admits the basin of attraction
${A(a_{\lambda})}=\{z\in {\mathbb{C}} :
\zeta^{n}_{\lambda}(z)\rightarrow a_{\lambda} \;\; \text{as} \; n
\rightarrow \infty \}$ for its attracting fixed point $a_{\lambda}$.
The following theorem gives the characterization of the Julia
set $J(\zeta_{\lambda})$ as the complement of basin of attraction $A(a_{\lambda})$:
\begin{thm} \label{thmcc}
Let $\zeta_{\lambda}\in \cal M$, $1<\lambda<\lambda^{*}$. Then, 
\begin{equation*}
J(\zeta_{\lambda})={\mathbb{\hat{C}}} \backslash { A(a_{\lambda})}
\end{equation*}
\end{thm}
\begin{proof}
By taking the fixed point $a_{\lambda}$ instead of $0$ and using Theorem~\ref{thm33}(c) instead of Theorem~\ref{thm33}(a), the proof of theorem follows on the lines of proof similar to that of Theorem~\ref{thmaa}.
\end{proof}
\begin{rem}
By Theorem~\ref{thm33}(c), for $1<\lambda<\lambda^{*}$,
$\zeta_{\lambda}^{n}(x) \rightarrow a_{\lambda}$ as $n\rightarrow\infty$
for $[(-\infty, -1)\cup (-1, 0) \cup (0, \infty)]\backslash T_{p}$.
Therefore,
it follows that the interval ${\mathbb{R}}\backslash\{T_{p}\cup\{0\}\}$
is contained in the basin of attraction ${ A}(a_{\lambda})$ for $1<\lambda<\lambda^{*}$.
\end{rem}

The following proposition shows that
the Fatou set of $\zeta_{\lambda}(z)$ for $\lambda=\lambda^{*}$ contains certain intervals of real
line:
\begin{prop} \label{thmdd}
Let $\zeta_{\lambda}\in \cal M$, $\lambda=\lambda^{*}$. Then, $F(\zeta_{\lambda})$ contains the
intervals $(-\infty, -1)$, $(-1, 0)\backslash T_{p}$, $(0, \sqrt{2})$ and
$(\sqrt{2}, \infty)$.
\end{prop}
\begin{proof}
By taking the fixed point $a_{\lambda}$ instead of $0$ and using
Theorem~\ref{thm33}(d) instead of Theorem~\ref{thm33}(b), the proof of
proposition follows on the lines of proof similar to that of Proposition~\ref{tpd13}.
\end{proof}
\noindent
Let $P_{a}$ be a set
 of attracting periodic orbits of $\zeta_{\lambda}(x)$. The following theorem describes the complex dynamics of $\zeta_{\lambda}(z)$ for
$z\in \mathbb{\hat{C}}$ and $\lambda >\lambda^{*}$ showing that
the Julia set $J(\zeta_{\lambda})$ contains the subset of real line
consisting of the complement of attracting periodic orbits of
$\zeta_{\lambda}(x)$: 
\begin{thm} \label{thmee}
Let $\zeta_{\lambda}\in\cal M$, $\lambda >\lambda^{*}$. Then,
 $J(\zeta_{\lambda})$
 contains the set  $\mathbb{R} \backslash P_{a}$.
 \end{thm}
 \begin{proof}
 By Remark~\ref{rem21}, for $\lambda >\lambda^{*}$, it follows that all points on $\mathbb{R}
 \backslash T_{p}$ either attract, or repel, or are chaotic. Therefore,
 Julia set $J(\zeta_{\lambda})$ contains the set $\mathbb{R} \backslash
 \{T_{p}\cup P_{a}\}$. Since pole and preimages
 of the pole are contained in $J(\zeta_{\lambda})$ and attracting orbits are
 contained in the Fatou set, it follows that $J(\zeta_{\lambda})$ contains the set
 $\mathbb{R} \backslash P_{a}$.
 \end{proof}

\begin{rem}
The results in Theorems~\ref{thmaa},~\ref{thmcc} and~\ref{thmee} are obtained for one parameter family of functions $\zeta_{\lambda}(z)$ when $\lambda$ is a real parameter. The ergodic properties of the Julia sets in this case and analogous investigations, when $\lambda$ is a complex parameter, is to follow in a subsequent work. 
\end{rem}

\section{Simulations and Comparisons} \label{algo44}
The characterizations of the Julia set $J(\zeta_{\lambda})$ of 
the function $\zeta_{\lambda}(z)$ in Theorems~\ref{thmaa}, \ref{thmcc}
and \ref{thmee} give the following  algorithm to  computationally
generate the images of the Julia set of $\zeta_{\lambda}(z)$:
{\bf (i)} Select a window W in the plane and divide W into $k\times k$
grids of width d.
{\bf (ii)} For the midpoint of each grid (i.e. pixel), compute the orbit
upto a maximum of $N$ iterations.
{\bf (iii)} If, at $i < N$, the modulus of the orbit is greater than some
given bound
M, the original pixel is colored red and the iterations are stopped.
{\bf (iv)} If no pixel in the modulus of the orbit ever becomes greater than
M, the original pixel is left as color according to number iterations.
Thus, in the output generated by this algorithm, the white points
represent the Julia set of $\zeta_{\lambda}(z)$ and the red points
represent the Fatou set of $\zeta_{\lambda}(z)$.
\par
Using the above algorithm, the Julia sets of a function
$\zeta \in {\cal M}$ for
$\lambda = 0.9$, $\lambda = 1.1$, $\lambda = 9.93$ and $\lambda = 9.94$
are generated  in the rectangular domains
$R = \{ z \in {\mathbb{C}} : -1.0 \le \Re(z) \le 1.0, \;\; -1.0 \le \Im(z) \le 1.0 \}$ upto $N=250$.
The Julia set of the function $\zeta_{\lambda}(z)$ for $\lambda = 0.9$
given by Fig.~\ref{juliasets4}(a) has the same pattern as those of the
Julia sets of $\zeta_{\lambda}(z)$
for any other $\lambda$ satisfying $0 < \lambda < 1$. This conforms to the result
of Theorem~\ref{thmaa}.  The nature of images of the Julia sets of
the function $\zeta_{\lambda}(z)$ have the same pattern
for all $\lambda$ satisfying $1< \lambda < \lambda^{*}$. This is visualized in
Fig.~\ref{juliasets4}(b) for $\lambda = 1.1$ and is in conformity to the result of
Theorem~\ref{thmcc}. The nature of image of the Julia set of
$\zeta_{\lambda}(z)$ for $\lambda = 9.94$ is given by Fig.~\ref{juliasets4}(d).
The Julia set of the function $\zeta_{\lambda}(z)$ for any other $\lambda >
\lambda^{*}$ remains the same as that of $J(\zeta_{\lambda})$ for
$\lambda = 9.94$. On comparison of Julia sets of $\zeta_{\lambda}(z)$
for $\lambda = 9.93 (<\lambda^{*} =(\sqrt{2}+1) e^{\sqrt{2}})$
and $\lambda = 9.94 (>\lambda^{*} =(\sqrt{2}+1) e^{\sqrt{2}})$
({\it Figs.~\ref{juliasets4}(c) and (d)}), it is observed that the phenomenon of chaotic burst
is not visible here. The reason probably
being just that the orbits are chaotic but these remain bounded
after crossing the parameter value.

\begin{figure}[h]
\centering
\subfloat{\includegraphics[angle=0, scale=.25]{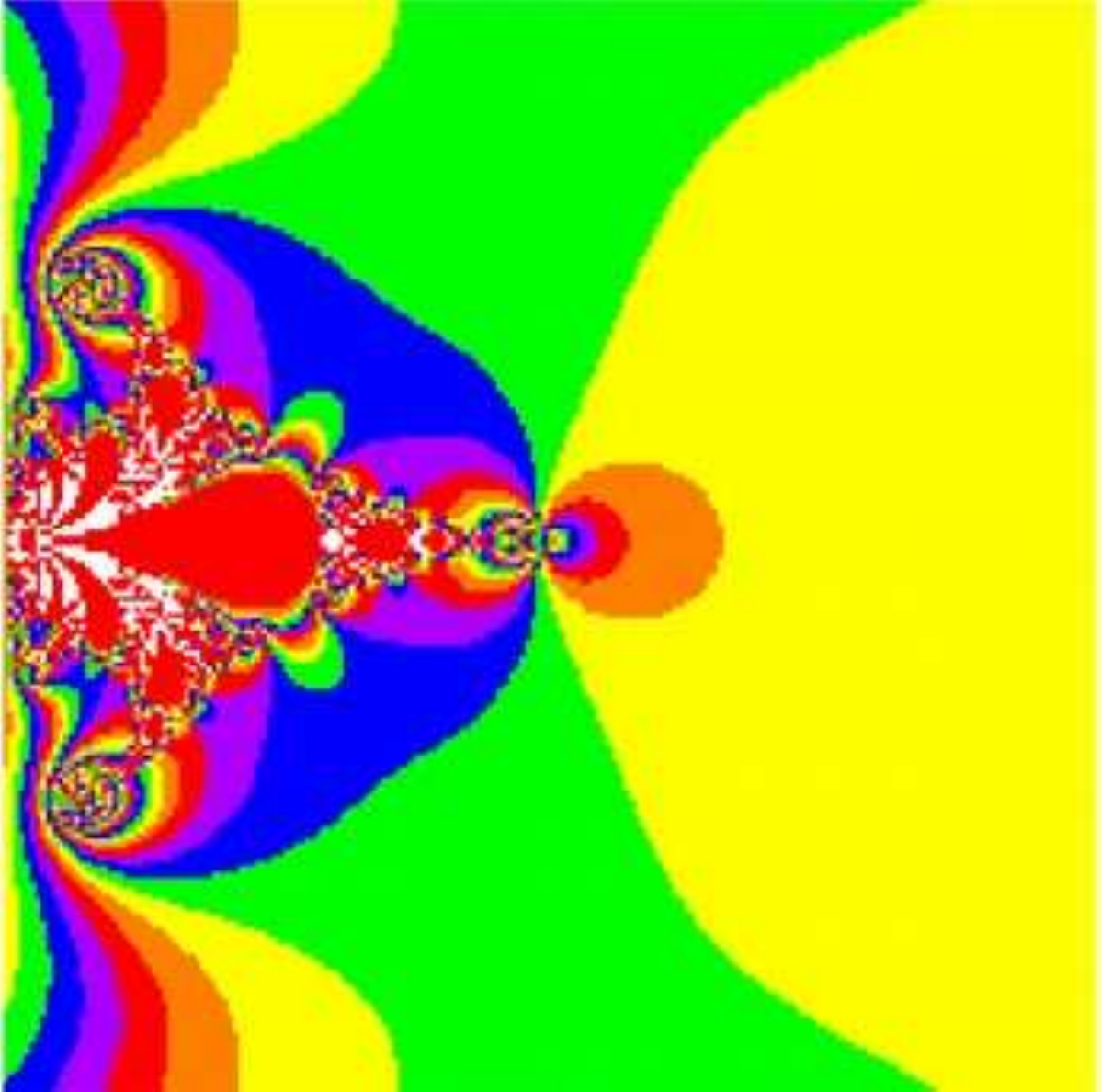}}  \hfill
\subfloat{\includegraphics[angle=0, scale=.25]{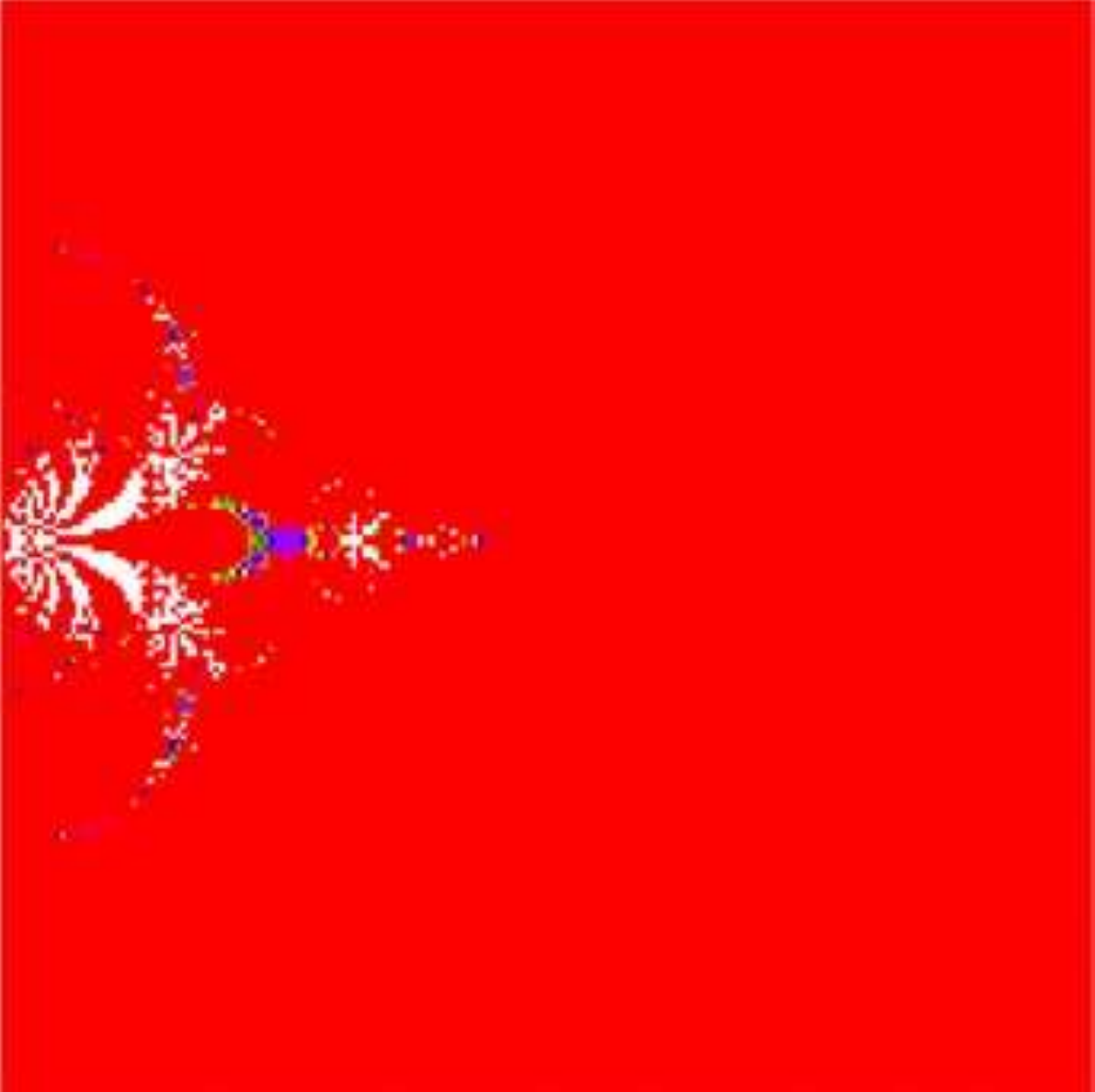}} \hfill
\subfloat{\includegraphics[angle=0, scale=.25]{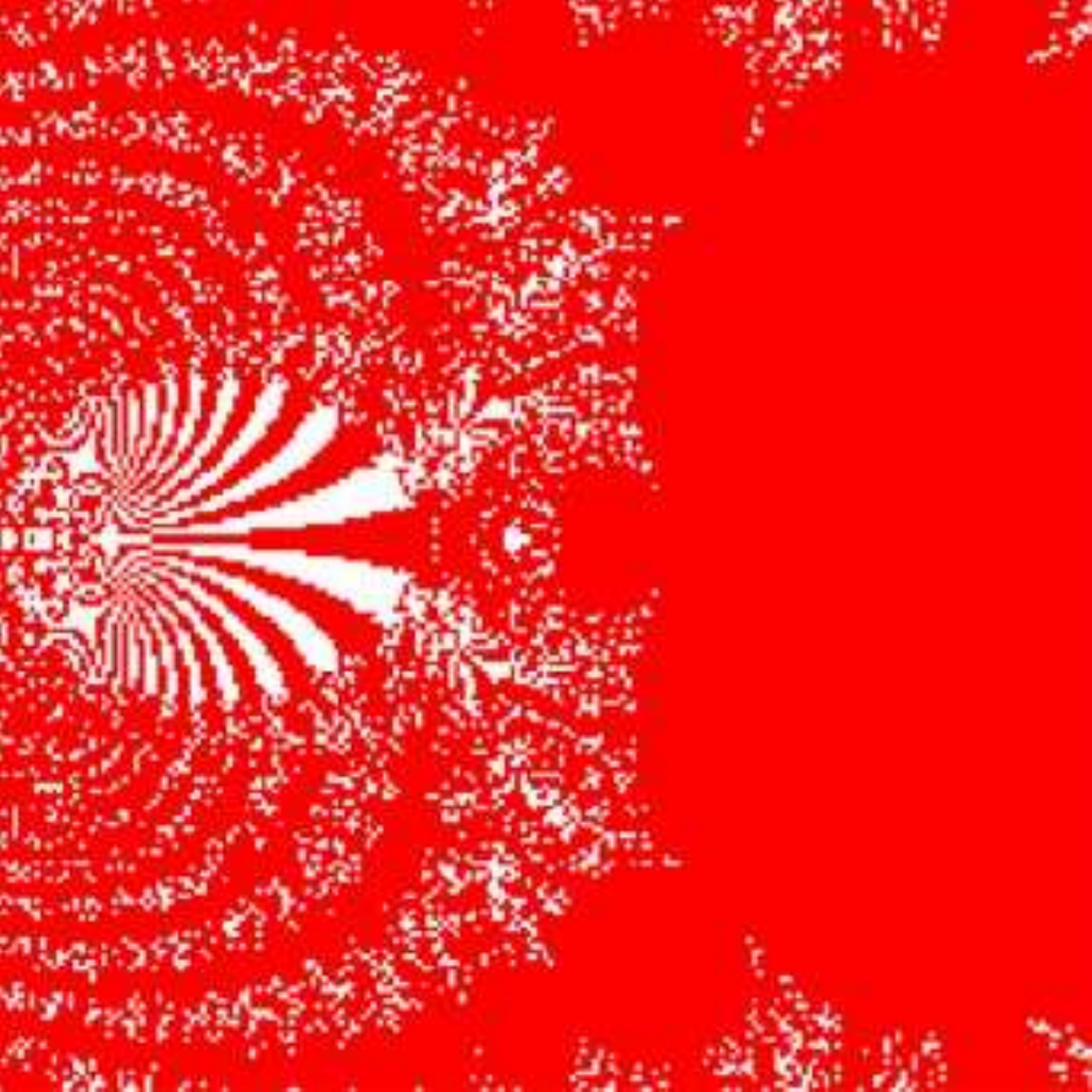}} \hfill
\subfloat{\includegraphics[scale=.25]{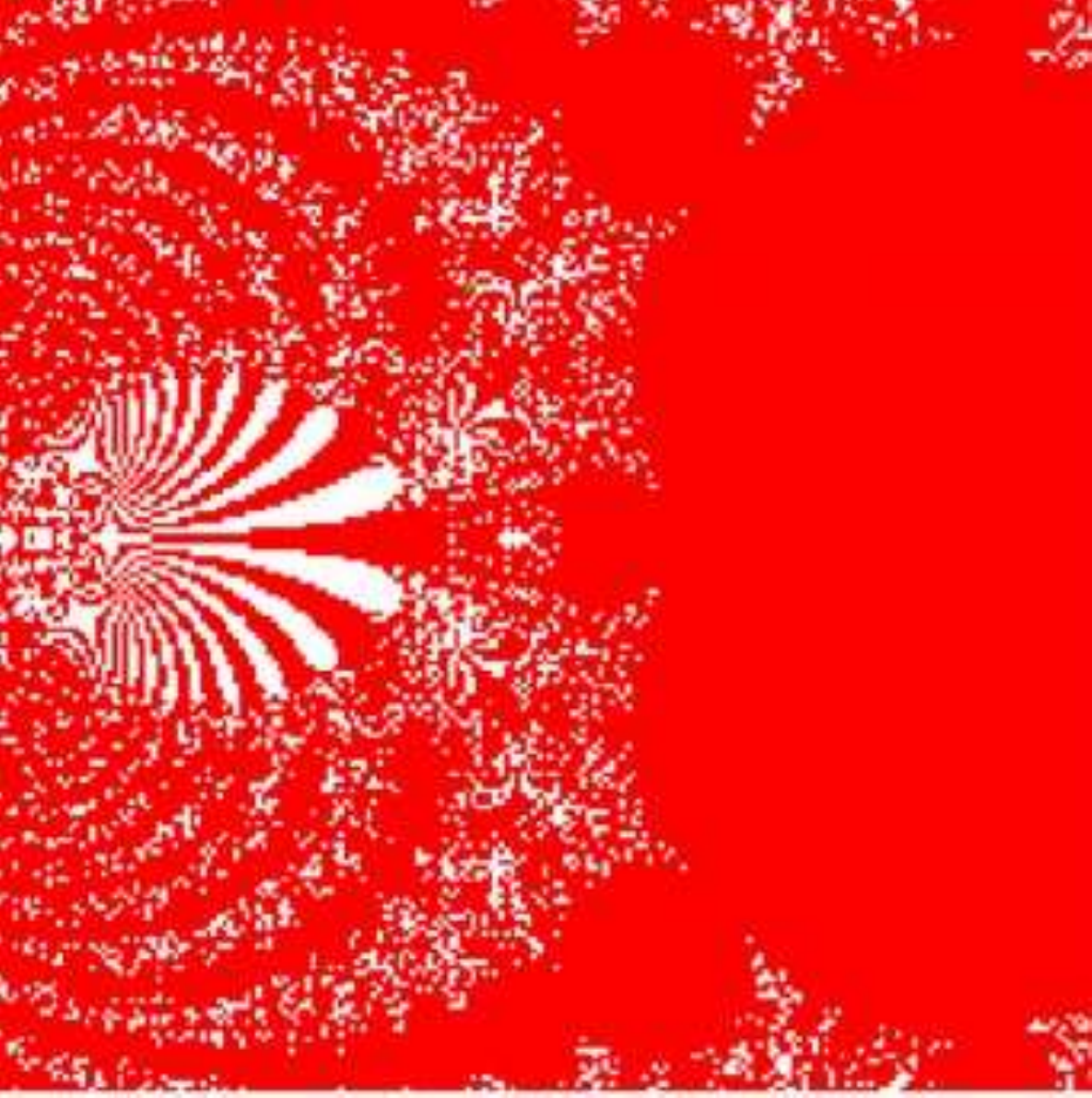}}
\caption{Julia sets of $\zeta_{\lambda}(z)$ for $\lambda=0.9$, $\lambda=1.1$, $\lambda=9.93$ and $\lambda=9.94$}
\label{juliasets4}
\end{figure}

Finally, the comparison between
dynamical properties are found here for the functions $\zeta_{\lambda}(z)$ with the dynamics of functions (i) $T_{\lambda}(z)=\lambda \tan z$, $\lambda \in
{\mathbb{\hat{C}}}\setminus \{0\}$ having polynomial Schwarzian
Derivative~\cite{devkeen1,devtanger,stallard94} (ii) $E_{\lambda}(z) = \lambda \frac{e^{z} -1}{z}$, $\lambda > 0$ having transcendental meromorphic Schwarzian Derivative~\cite{gpkmgpp1} and (iii)
$f_{\lambda}(z)=\lambda f(z)$, $\lambda>0$, where $f(z)$ has only real critical values, and having rational Schwarzian
Derivative~\cite{msgpk}. Based on the results obtained in earlier sections, it is observed that although
the nature of the functions in the families under comparison is
widely different,  the dynamical behaviour of the functions in all
these families have an underlying differences in the nature of dynamics of functions in our family obtained in earlier sections from that of the family of functions considered earlier, bifurcation occurs at different number of parameter values in our family of functions than the number of parameter values at which bifurcation occurs in the families of functions $T_{\lambda}(z)$,  $E_{\lambda}(z)$ and $f_{\lambda}=\lambda f(z)$, the Julia set does not contain whole real line in our family of functions but the Julia sets for functions in the families $T_{\lambda}(z)$,  $E_{\lambda}(z)$ contain the whole real line and Julia set is not the whole complex plane for any value of parameter $\lambda$ for our family of functions but it is whole complex plane for  the family of functions $T_{\lambda}(z)$ for the parameter value $\lambda=i\pi$. Moreover,
these families have an underlying similarity in dynamical behaviour, all singular values of these families are bounded, the Fatou set equals the basin of attraction of the real attracting fixed point,  the Julia set is the closure of escaping points and; Herman rings and wandering domains do not exist.

\end{document}